\documentclass[10pt]{article}

\usepackage[applemac]{inputenc}

\usepackage{amsmath,amsthm}
\usepackage{amssymb} 
\usepackage{enumitem}
 
\usepackage{graph icx}
\usepackage[T1]{fontenc} 
\usepackage{hyperref} 
\usepackage{hyperref} 
\usepackage{xcolor}
	\hypersetup{
 colorlinks,
 linkcolor={red!50!black},
 citecolor={blue!50!black},
 urlcolor={blue!80!black}
}

\newtheorem{theorem}{Theorem}[section]
\newtheorem{corollary}[theorem]{Corollary}
\newtheorem{lemma}[theorem]{Lemma}
\newtheorem{proposition}[theorem]{Proposition}

\theoremstyle{definition}
\newtheorem{definition}[theorem]{Definition}
\newtheorem{remark}[theorem]{Remark}
\newtheorem{example}[theorem]{Example}

\numberwithin{equation}{section} 
 
\newcommand {\Q}{{\mathbb{Q}}}
\newcommand {\C}{{\mathbb{C}}}

\newcommand {\PP}{{\mathbb{P}}}
\newcommand {\R}{{\mathbb{R}}}
\newcommand {\Z}{{\mathbb{Z}}} 
\newcommand {\K}{{\mathbb{K}}} 

\newcommand{\rmh}{\mathrm {h}}
\newcommand{\rmH}{\mathrm {H}}

\newcommand{\rmM}{\mathrm {M}}
\newcommand{\rme}{\mathrm {e}}

\newcommand{\Bin}{\mathrm {Bin}}

\newcommand{\calA}{\mathcal {A}}

\newcommand{\calE}{\mathcal {E}} 
\newcommand{\calF}{\mathcal {F}}
\newcommand{\calG}{\mathcal {G}}

\newcommand{\calR}{\mathcal {R}}
\newcommand{\calX}{\mathcal {X}}

\newcommand{\GL}{\mathrm {GL}}

\def\boxit#1#2{\setbox1=\hbox{\kern#1{#2}\kern#1}%
\dimen1=\ht1 \advance\dimen1 by #1 \dimen2=\dp1 \advance\dimen2 by #1
\setbox1=\hbox{\vrule height\dimen1 depth\dimen2\box1\vrule}%
\setbox1=\vbox{\hrule\box1\hrule}%
\advance\dimen1 by .4pt \ht1=\dimen1
\advance\dimen2 by .4pt \dp1=\dimen2 \box1\relax}

\begin{document}


\baselineskip=14pt

\author{\'Etienne Fouvry\\
Univ. Paris--Saclay, CNRS \\
Laboratoire de Math\' ematiques d'Orsay \\
91405 Orsay, France \\
E-mail: Etienne.Fouvry@universite-paris-saclay.fr
\and 
Michel Waldschmidt\\
Sorbonne Universit\' e \\
CNRS, IMJ--PRG \\
75005 Paris, France\\
E-mail: michel.waldschmidt@imj-prg.fr}

\renewcommand{\thefootnote}{\arabic{footnote}}
\setcounter{footnote}{0}

\begin{center}
\Large\bf 
{\bf Number of integers represented by
 families of 
 \\
binary forms III: fewnomials
 } 
\\

\bigskip

{\sc \small \'Etienne Fouvry \&\ Michel Waldschmidt}

\end{center}
 
 \bigskip
 
\hfill 
{\em 
To J\'anos Pintz for his Seventy Fifth Birthday.}

 
 \bigskip
   
\noindent
{\bf Abstract.} In a series of papers we investigated the following question: 
given a family $\calF$ of binary forms having nonzero discriminant and integer coefficients, for each $d\geqslant 3$, we estimate the number of integers $m$ with $|m|\leqslant N$ which are represented by an element in $\calF$ of degree $\geqslant d$. Under suitable assumptions, asymptotically as $N\to\infty$, the main term in the estimate is given by the forms in $\calF$ having degree $d$ (if any), while the forms of degree $>d$ contribute only to the error term. 
The present text is devoted to fewnomials
\[
a_0X^{kr}+a_1X^{k(r-1)}Y^k+\cdots +a_{r-1}X^kY^{k(r-1)}+a_rY^{kr}
\]
with fixed $r\geqslant 1$ and varying $k,a_0,a_1,\dots,a_r$.

 \bigskip



 \section{Introduction} 
 This is the fourth text of a series of papers devoted to the study of the set of integers which are represented by some forms belonging to a family. In the first one \cite{FW0} we investigated the case of the family of cyclotomic forms. In the second and third texts \cite{FW1, FW2} we considered in particular families of binomial binary forms $aX^d+bY^d$, with varying $a,b,d$, such that $a$ and $b$ are of any sign and $d \geqslant 3$. We now are concerned with fewnomials
\[
a_0X^{kr}+a_1X^{k(r-1)}Y^k+\cdots +a_{r-1}X^kY^{k(r-1)}+a_rY^{kr}
\]
with fixed $r\geqslant 1$ and varying $k,a_0,a_1,\dots,a_r$. Let $\mathcal F$ be some (suitably defined) family of such fewnomials. When $r=1$ we recognize a family of binomial forms. 
Like in \cite{FW0}, \cite{FW1} and \cite{FW2}, we are interested by the set of integers represented by some form of the family $\mathcal F$. Our method rests on a lower bound for linear forms in logarithms (see Proposition \ref{Prop:9.22} below), on a study of the group of automorphisms of the forms in $\mathcal F$ 
and on the non existence of isomorphism exchanging two distinct forms of $\mathcal F$
(see Propositions \ref{central724} and \ref{central}).

The results of the present paper apply to families of trinomial binary forms 
\[
aX^d+cX^eY^{d-e} +bY^d,
\] 
with varying rational integers $a,b,c,d,e$, $1\leqslant e<d$, $abc\not=0$ when the quotients $e/d$ belongs to a finite set depending on the family. We will pursue the study of more general families of trinomial binary forms in a forthcoming paper \cite{FW6+}. We will see that the famous {\it Conjecture abc } has a dramatic impact on the qualities of the results. We will consider families of definite positive forms in another one \cite {FW6}. 

In order to present our results, we give a list of definitions and notations.

\subsection{About general binary forms}

 Let $d\geqslant 3$ be an integer. For $\K = \Z,\, \Q, \R$ or $\C$, let $\Bin  (d, \K)$ be the set of binary forms
 $F= F(X,Y)$ with degree $d$, with coefficients in $\K$ and with discriminant different from zero. If $F$ belongs to 
 $\Bin (d, \K)$ and if \begin{equation}\label{abcd}
 \gamma= \begin{pmatrix} u_1& u_2 \\ u_3 &u_4\end{pmatrix}
\end{equation}
 belongs to ${\GL}(2, \K)$, we denote by
 $F\circ \gamma$ the binary form defined by 
 $(F\circ \gamma) (X, Y) = F(u_1X+u_2Y, u_3X+u_4Y),
 $
 after the associated linear change of variables. The form $F\circ \gamma$ belongs to $\Bin  (d, \K)$. 
 Two forms $F$ and $G$ in $\Bin  (d, \K)$ are called {\em $\K$--isomorphic} if there exists $\gamma$ in ${\GL}(2, \K)$ such that $F\circ \gamma=G$.
 
 If $\gamma$ is such that $F\circ \gamma =F$, we say that $\gamma$ is an {\em automorphism of $F$}. The set of these automorphisms is a group denoted by ${\rm Aut}(F, \K)$.
 We have $-{\rm Id}\in {\rm Aut}(F, \K)$ if and only if $d$ is even.
 
 When $F$ belongs to ${\rm Bin } (d, \Z)$ we denote by $C_F$ the constant 
 \begin{equation}\label{defCF} C_F:=A_FW_F,
 \end{equation}
 attached to $F$. It is defined and thoroughly studied in \cite[Theorem 1.2]{SX}: according to \cite[Theorem 1.1]{SX}, the number of $m\in\Z$ in the interval $[-N,N]$ which are represented by $F$ is equal to 
 \begin{equation}\label{Equation:CFN2/d}
 C_FN^{2/d}+O_F (N^{(2/d) -\kappa_d}),
 \end{equation}
 where $\kappa_d >0$ is an effective constant only depending on $d$, uniformly for $N \geqslant 1$. The constant $A_F$ is the area of the fundamental domain attached to $F$: 
 \begin{equation}\label{defAF}
 A_F:=\iint_{ 
 \vert F(x,y) \vert \leqslant 1} dx\, dy,\end{equation}
 and $W_F$ is a positive rational number the delicate definition of which is based on the denominators of the entries of the matrices in 
 ${\rm Aut}(F, \Q)$ (see \cite[Theorem 1.2]{SX}). For the purpose of our present work, we will only retain the following values of $W_F$ 
 \begin{equation}\label{valuesofWF}
 W_F=
 \begin{cases}
 1 & \text{ if } {\rm Aut}(F, \Q) =\{ {\rm Id}\}, \\
 1/2 & \text{ if } {\rm Aut}(F, \Q) = \{ \pm {\rm Id}\},\\
 1/4 & \text{ if } {\rm Aut}(F, \Q) = {\mathbf D}_2,
 \end{cases}
 \end{equation}
 where ${\mathbf D_2}\subset{\GL}(2, \Z)$ is the group with four elements
 \[
 {\mathbf D_2} =\left\{ \pm {\rm Id},\, \pm \begin{pmatrix}1 & 0\\ 0 & -1 \end{pmatrix}\right\}.
 \]
 Motivated by the example of forms \eqref{balancedtrinomial} with even $k$ (see below), we say that a binary form $F(X,Y)$ is a {\em binary form with squared arguments} when there exists a binary form $H$ such that $F(X,Y) =H(X^2, Y^2)$. Necessarily $\deg F $ is even and ${\rm Aut} (F, \Q)$ contains ${\mathbf D}_2$.
 
 The following definitions concerns family of binary forms containing essentially distinct forms. 
 \begin{definition} \label{definition:dilation homograph free}
 Let $\K$ as above and let $\mathcal E$ be a set of binary forms of any degree $d\geqslant 3$ with coefficients in $\K$. 
\begin{enumerate}
\item We say that $\mathcal E$ is {\em $\K$--dilation--free}, if for any $F$, $G$ in $\mathcal E$ and any $u$ and $v$ in $\K^\times$, 
the condition $F(uX,vY)=G(X,Y)$ implies $F=G$. 
 
\item We say that $\mathcal E$ is {\em $\K$--homography--free}, if the following 
 condition holds:
For any distinct forms $F$ and $G$ in $\mathcal E$ we have the equality
 \[
 \{ \gamma \in {\GL}(2, \K) : F= G\circ \gamma \} =\emptyset.
 \]
\item A form $F \in \mathcal E$ is called {\em $\K$--rigid}
 if we have the equality 
 \begin{equation}\label{299}
{\rm Aut}(F, \K)
 = 
 \begin{cases}\{ \lambda\, {\rm Id} : \lambda \in \K,\, \lambda^{\deg F} =1\} \text{ if } F \text{ is not a binary form}\\
 \hfill\text{ with squared arguments,}
 \\
 \\
\displaystyle {\bigcup_{\lambda \in \K,\quad \lambda^{\deg F} =1}\lambda\cdot {\mathbf D}_2}\text{ otherwise.}\\
 \end{cases}
\end{equation}
 \end{enumerate} 
\end{definition} 
 
 In section \S\ref{SS:HomographyFreeSets}, we give examples of sets $\mathcal E$ which are $\K$--homography--free. 

A set which is $\K$--homography--free is also $\K$--dilation--free. 
 If $\mathcal E$ is a set which is $\K$--homography--free (resp. $\K$--dilation--free), so is any subset of $\mathcal E$.

From \eqref{299} it follows that if $F\in \Bin  (d, \Q)$ is a $\Q$--rigid form, we then have 
 \[
 {\rm Aut}(F, \Q)=
\begin{cases}
\{\mathrm{Id}\} &\text{
if $d$ is odd,}
\\
\{\pm \mathrm{Id}\} &\text{
if $d$ is even}
\end{cases}\text{ if $F$ is not a binary form with squared arguments }
 \] 
 and ${\rm Aut}(F, \Q)= \mathbf D_2$ otherwise. 
Therefore, for a $\Q$--rigid form $F$, we have 
 \begin{equation}\label{=1/2ou1/4}
 W_F =
 \begin{cases}
1/(2,d)
& \text{ if } F \text{ is not a binary form with squared arguments}
 \\
 1/4
 & \text{ otherwise }
\\ 
\end{cases}
\end{equation}
 by \eqref{valuesofWF}.

 \subsection{About family of binary forms.}
 
 \begin{definition} \label{defFamily}
 Let $\K = \Z,\, \Q, \R$ or $\C$. Let $\mathcal F$ be a set of binary forms. We say that $\mathcal F$ is a {\it $\K$--family of binary forms} 
 if the two following conditions hold \begin{itemize}
 \item $\mathcal F \subset \bigcup_{d\geqslant 3} \Bin (d, \K)$ 
 \item for every $d\geqslant 3$, the set $\mathcal F \cap {\rm Bin }( d, \K) $ is finite.
 \end{itemize}
 \end{definition}
 For $d\geqslant 3$, set 
 \begin{equation}\label{defFd}\mathcal F_d:= \mathcal F \cap {\rm Bin }( d, \K).
 \end{equation} When the family $\calF$ is given and when the integer $d$ is $\geqslant 3$, 
the integer $d^\dag$ is defined by the formula
\[
d^\dag :=
\begin{cases}
\inf\{ d' : d' >d \text{ such that } \calF_{d'} \ne \emptyset\}
&\hbox{if there exists $d'>d$ such that $\calF_{d'} \ne \emptyset$,}
\\
\infty
& \hbox{if $\calF_{d'} = \emptyset$ for all $d'>d$.}
\end{cases}
\] 
When $\K= \Z$ and when $\mathcal F$ is fixed, we are interested in describing the value set of $\mathcal F$ defined as the union of all the images $F (\Z^2)$ for $F\in\mathcal F$. So 
we introduce the two sets 
\begin{multline} \label{defCalG}
\calG_{\geqslant d}(m)=
\Big\{(x,y,F)\;
 \mid
m=F(x,y)\; \hbox{ with } 
 F\in \mathcal F, \deg F \geqslant d \\ (x,y)\in\Z^2 \; \hbox{ and } 
 \max\{|x|,|y|\}\geqslant 2 \Big\}
 \end{multline}
 and 
\[
\calR_{\geqslant d}=\left\{m\in\Z\; \mid\; \calG_{\geqslant d}(m)\not=\emptyset \right\}.
\]
The assumption $\max\{|x|,|y|\}\geqslant 2$ is natural: the coefficient $a_0=F(1,0)$ of $F$ is likely to take infinity many values $m$, some of which may be repeated infinitely often 
(see Remark 1.2 in \cite{FW2}), a situation which would yield for the modified $\calG_{\geqslant d}(m)$ an infinite set.

 For $N$ a positive integer, we introduce
 \begin{equation}\label{Equation:RdN}
 \calR_{\geqslant d}(N)=\calR_{\geqslant d}\cap[-N,N].
 \end{equation}
 
\subsection{About binary fewnomials} \label{S:Fewnomial}

We firstly define a family of binary fewnomials.
\begin{definition}\label{fewnomial} Let $r\geqslant 1$ be an integer. For every $k\geqslant 3/r$ let $\mathcal E_k$ be a finite subset of $\Z^{r+1}$
such that, for every $\boldsymbol a =(a_0, \dots, a_r) \in \mathcal E_k$, one has 
\begin{enumerate}
\item $a_0 a_r\not= 0$,
\item the discriminant of the polynomial $a_0T^r+ \cdots +a_r$, is different from zero.
\end{enumerate} For every $k\geqslant 3/r$ and 
for every $\boldsymbol a\in \mathcal E_k$, let $F=F_{k,\boldsymbol a}(X,Y)$ be the binary form
\begin{equation}\label{278bis}
F_{k, \boldsymbol a} (X,Y)= a_0 X^{kr} +a_1 X^{k(r-1)} Y^k +\cdots +a_{r-1}X^k Y^{k(r-1)} + a_r Y^{kr}.
\end{equation}
Then the set $\mathcal F= \mathcal F_{\boldsymbol{\mathcal D}}$ defined by 
 \[
 \mathcal F_{\boldsymbol{\mathcal D}}:=
\left\{ F_{k, \boldsymbol a} : k\geqslant 3/r, \, \boldsymbol a \in \mathcal E_k \right\},
\]
is called the family of {\em binary fewnomials attached to the data } 
\begin{equation*} 
{\boldsymbol {\mathcal D}}=(r, (\mathcal E_k)).
\end{equation*}
\end{definition}
Let $\mathcal F$ as in Definition \ref{fewnomial}. Then the degree of every $F\in \mathcal F$ is divisible by $r$
and greater than $2$. The discriminant of $F$ is different from zero. We define the {\em height of $F$} by the formula.
\begin{equation}\label{Equation:calA}
 \calA(F):=\max\{\vert a_0\vert,\, \vert a_1\vert,\dots,\, \vert a_r\vert \}
\end{equation}
and the modified height 
\[
\calA^\star(F):=\max\{2,\calA(F)\}
\] 
which naturally appears in some formulas (for instance in Corollary \ref{Cor:H(xd,yd)}).

By the definition \eqref{defFd} we have the 
 decomposition
 \[
\mathcal F = \bigcup_{k\geqslant 3/r} \mathcal F_{kr}.
\]
The number of elements in $\mathcal F_d$
is less than
 \begin{equation}\label{Equation:New1}
 \max_{F\in \mathcal F_d}
\left(2\, \calA ^\star(F)+1\right)^{r+1}.
\end{equation}
If $F$ is given by \eqref{278bis},
we have the equality
 \[
F(X,Y) =Y^{kr}h\left( \left(X/Y
\right)^k\right),
\]
where $h(T)$ is the polynomial $a_0 T^r +\cdots + a_{r-1} T + a_r$. This point of view will be exploited in the proof of Corollary \ref{Cor:H(xd,yd)}.

\subsection{Some examples}

Let $\mathcal F = \mathcal F_{\boldsymbol {\mathcal D}}$ be the family of binary fewnomials attached to the data $\boldsymbol {\mathcal D}$ as in Definition \ref{fewnomial}.
The number of monomials appearing in each form $F\in\calF_d$ is at most $ r+1$. In particular, when $r=1$, the forms $F\in\calF_d$ are binomial binary forms 
\begin{equation}\label{Equation:binomialForms}
aX^d+bY^d,
\end{equation}
(cf. \cite[Corollary 1.3]{SX}, \cite{FW2}), while if $r\geqslant 2$ and if there exists $s$ in the interval $1\leqslant s\leqslant r-1$ such that each ${\boldsymbol a} = (a_0,a_1,\dots,a_r)\in\calE_k$ satisfies $a_j=0$ for $j\not\in\{0,s,r\}$, then the forms $F\in\calF_d$ are {\em trinomial binary forms }
\[
aX^{kr}+cX^{ks}Y^{k(r-s)} +bY^{kr}.
\] 
For example with $r=2$ and $s=1$ the family that we are considering is the family of 
{\it balanced trinomial binary forms } 
\begin{equation}\label{balancedtrinomial}
aX^{2k}+cX^kY^k+bY^{2k}.
\end{equation}
These trinomial forms will be studied in \cite{FW6+}.

As in \cite{FW1}, \cite{FW2} we are interested in the asymptotic description of the set of integers which are represented by some form $F$ of the family of binary fewnomials $\mathcal F = \mathcal F_{\boldsymbol {\mathcal D}}$, with a fixed $r\geqslant 2$, in particular we study the counting function 
\[
\sharp \calR_{\geqslant d}(N),
\] 
associated to $\mathcal F_{\boldsymbol{ \mathcal D}}$ (see the Definition \eqref{Equation:RdN}), where as usual $\sharp E$ the number of elements of a finite set $E$.
Our results 
will require the constant $\vartheta_d\, (<2/d)$ which is defined in \cite[(2.1)]{FW1} by the formula:
 \begin{equation*} 
 \vartheta_d=
 \begin{cases}
 \displaystyle
 \frac {24\sqrt 3+73}{60\sqrt 3 +73}=\frac {2628\sqrt 3-1009}{5471}=0.6475\dots &\text{ {for} } d=3,\\
\displaystyle
 \frac {2\sqrt d +9}{4 d \sqrt d -6\sqrt d +9} & \text{ {for} } 4\leqslant d\leqslant 20,\\
 \displaystyle
 \frac 1{d-1} & \text{ {for} } d\geqslant 21.
 \end{cases}
 \end{equation*}

\subsection{The main result} 

Recall the Definitions \ref{defFamily} for $d^\dag$ and \ref{fewnomial} for ${\boldsymbol{ \mathcal D}} $ and $\mathcal F_{\boldsymbol{ \mathcal D}}$, and Notations
\eqref{defCF} for $C_F$, \eqref{defAF} for $\calA(F)$, 
\eqref{defCalG} for $\calG_{\geqslant d}(m)$,
\eqref{Equation:RdN} for $ \calR_{\geqslant d}(N)$.

 If $F$ is a binary form we put $A^*_F :=A_F/2$ if $F$ is a form with squared arguments, 
 and $A^*_F :=A_F$ otherwise. 
The next result will be proved in \S \ref{Subsection:ProofPrincipal}.

 \begin{theorem} \label{Th:principal}
Let $\epsilon>0$ and let $r\geqslant 2$ be an integer. Let $\mathcal F= \mathcal F_{\boldsymbol{ \mathcal D}}$ be the family of binary fewnomials attached to the system 
of data ${\boldsymbol{ \mathcal D}} =(r, (\mathcal E_k))$. 
There exists a constant $\eta>0$ which depends only on $\epsilon$ and $r$ with the following property. 
Assume that there exists $d_0\geqslant 3$ such that, for all $d\geqslant d_0$,
\begin{equation}\label{Equation:majorationCalAeta}
\max_{F\in\calF_d}\calA^\star(F)\leqslant \exp(\eta d/\log d).
\end{equation}
Then 
 
(a) For all $m\in\Z\smallsetminus\{-1,0,1\}$ and all $d\geqslant 3$ which is a multiple of $r$, the set $\calG_{\geqslant d}(m)$ is finite. Moreover, for all $d\geqslant d_0$ which is a multiple of $r$ and all $\epsilon>0$, there exists a constant $c$ depending only on $r,d,\epsilon$ such that, for $|m|\geqslant 2$,
\[
\sharp \calG_{\geqslant d}(m)\leqslant c |m|^{(1/d)+\epsilon }.
\]

(b) 
Assume that $\mathcal F$ is a $\Q$--homography--free set. Then 
for all $d\geqslant 3$ which is a multiple of $r$, we have, for $N\to\infty$,
\begin{equation}\label{Equation:calRdN}
\sharp \calR_{\geqslant d}(N)=
\left(\sum_{F\in\calF_d} C_F\right) N^{2/d}+ O_{\epsilon,r,d} \Bigl(N^{\max\{\vartheta_{d}+\epsilon,2/d^\dag\}} \Bigr).
\end{equation}

(c) 
Assume that $\mathcal F$ is a $\Q$--homography--free set of $\Q$--rigid forms. Then 
for all $d\geqslant 3$ which is a multiple of $r$, we have, for $N\to\infty$,
\[
\sharp \calR_{\geqslant d}(N)= \frac 1{(d,2)}
\left(\sum_{F\in\calF_d} A^\star_F\right) N^{2/d}+ O_{\epsilon,r,d} \Bigl(N^{\max\{\vartheta_{d}+\epsilon,2/d^\dag\}} \Bigr).
\]
\end{theorem}

We will prove this result with $\eta=\epsilon (2^{80}3^{15}r^{4r})^{-1}$, corresponding to a value for $\mu$ given by \eqref{Equation:mu} with $\lambda=2+\epsilon$. 

 The proof of Theorem \ref{Th:principal} is given in \S \ref{Section:ProofPrincipal}: we first recall 
 the definition of a regular family \cite[Definition 2.2]{FW2} (Definition \ref{def:Regular}) and the statements of \cite[Theorem 2.6]{FW2} and of \cite[Corollary 9.22]{W} (Theorems \ref{Th:VFW2th1.6} and Proposition \ref{Prop:9.22} respectively). These tools allow us to prove the asymptotic estimate (Theorem \ref{Th:majorationasymptotique}) which is required for checking the conditions of a regular family.
 
 The rest of the paper is devoted to giving examples of families satisfying the assumptions of Theorem \ref{Th:principal}. These examples are stated in \S \ref{Section:Examples}, Theorems \ref{486} and \ref{527}. The main purpose of \S \ref{Section:tools} is to study isomorphisms among two binary fewnomials. The technical arguments are the proofs of Propositions \ref{central724} and \ref{central}. The proofs of the Corollaries \ref{Corollary2star} and \ref{Corollary2} are given in \S \ref{Proofs486et2star} and \S \ref{S:Proof527} respectively.
 
 \section{Proof of Theorem \ref{Th:principal}}\label{Section:ProofPrincipal}

\subsection{Regular families}

The next definition is Definition 2.2 of \cite{FW2}. 

\begin{definition}\label{def:Regular}
An infinite family $\mathcal F$ of binary forms as in Definition \ref{defFamily} with coefficients in $\Z$ is called {\em regular} if the following properties are satisfied:
\\
(i) Two forms in the family $\mathcal F$ are $\Q$--isomorphic if and only if they are equal. 
\\
(ii) There exists an integer $A>0$ such that for all $\epsilon>0$, there exist two positive integers $N_0=N_0(\epsilon)$ and $d_0=d_0(\epsilon)$ such that, for all $N\geqslant N_0$, the number of integers $m\in [-N,N]$ for which there exists $d\in\Z$, $(x,y) \in \Z^2$ and $F\in \calF_d$ with
\[
d\geqslant d_0,\quad 
 \max\{\vert x\vert, \vert y \vert\} \geqslant A\quad\hbox{and }\quad F(x,y)=m
\] 
is at most $N^\epsilon$. 
\end{definition}

We also borrow the following notation from \cite{FW2}:
\[
 \begin{aligned}
{\calR}_{\geqslant d} \left (\calF, N, A\right) := \sharp\, \bigl\{ m: 0\leqslant \vert m \vert \leqslant N, \, \text{ there is } \, F\in \calF \text{ with } \deg F \geqslant d
\\
\text{ and } (x,y)\in \Z^2 \text{ with } \max \{\vert x \vert, \vert y \vert\} \geqslant A, \text{ such that } F(x,y) =m
\bigr\}. 
\end{aligned}
\]

Here is the statement of \cite[Theorem 2.6]{FW2}.

\begin{theorem}\label{Th:VFW2th1.6}
Let $\calF$ be a regular family of distinct binary forms in the meaning of Definition \ref{def:Regular}. 
 Then for every $d\geqslant 3$ and every positive $\epsilon$, the quantity
 $
 {\calR}_{\geqslant d} \left (\calF, N, A\right)$ satisfies
\[
\calR_{\geqslant d} (\calF, N, A) = \left( \sum_{F\in \calF_d} A_F W_F \right) \cdot N^{2/d} + 
O_{\calF,A,d, \epsilon}\bigl( N^{\max\{ \vartheta_d+\epsilon, 2/d^\dag\} }\bigr), 
\]
uniformly as $N \to\infty$.
\end{theorem}

 \subsection{Diophantine tool}

Our main tool for the proof of Theorem \ref{Th:principal} is an asymptotic estimate (Theorem \ref{Th:majorationasymptotique}) which we are going to deduce from a lower bound (Proposition \ref{Prop:9.22}) arising from the theory of linear forms in logarithms, namely \cite[Corollary 9.22]{W}. 

Using the notations of \cite[\S 5]{FW2}, we denote by $\rmH$ the absolute height, by $\rmh$ the absolute logarithmic height and by $\rmM$ the Mahler's measure; for a rational number written in its irreducible form as $p/q$, we have
\[
\rmH(p/q)=\rmM(p/q)=\max\{|p|,q\}, \qquad
\rmh(p/q)= \log\max\{|p|,q\}.
\] 

\medskip
\begin{proposition} \label{Prop:9.22}
Let $K$ be a number field of degree $\leqslant D$, $\alpha_1, \alpha_2$ nonzero elements of $K$, $b_1,b_2$ positive integers, $A_1,A_2,B$ positive real numbers. Assume, for $j=1,2$, 
\[
B\geqslant \max \{\rme,\; b_1,\; b_2\},
\quad 
\log A_j\geqslant \max\left\{\frac 1 D,\; \rmh(\alpha_j)\right\}.
\]
If $ \alpha_1^{b_1}\alpha_2^{b_2}\not=1$, then 
\[ 
|\alpha_1^{b_1}\alpha_2^{b_2}-1|\geqslant \exp
\left\{ 
-C(\log B)(\log A_1)(\log A_2)D^4 \max\{1,\log D\}
\right\}
\] 
where $C=2^{79}3^{15}$. 
\end{proposition}

\medskip
When $D=1$, we recognize \cite[Proposition 5.1]{FW2}.

Proposition \ref{Prop:9.22} follows from \cite[Corollary 9.22 p. 308]{W} with the constant $C(m)$ of \cite[p. 252]{W}.

\bigskip

The next lemma gives an upper bound for the absolute logarithmic height $\rmh$ of an algebraic number in terms of its usual height. 

 \begin{lemma}\label{Lemme:hVersusH}
Let $\theta$ be an algebraic number which is root of a nonzero polynomial of degree $r$ having integer coefficients bounded by $H$. Then
\begin{equation*} 
 \rme^{\rmh(\theta)}\leqslant \sqrt{r+1} H.
\end{equation*}
 \end{lemma}
 
\begin{proof} 
We use the results of \cite[Chap.~3]{W} (see the proofs and notations of Lemmas 3.10 and 3.11). For $h\in\C[X]$ a polynomial of degree $r$, the coefficients of which have moduli $\leqslant H$, we have
\[
\rmM(h)\leqslant \sqrt{r+1}H.
\]
If $h\in\Z[X]$ and if $f$ is a factor of $h$ in $\Z[X]$, we have 
$\rmM(h/f)\geqslant 1$, hence $\rmM(f)=\rmM(h)/\rmM(h/f)\leqslant \rmM(h)$. 
Further, is $f$ is irreducible and if $\theta$ is a root of $f$, then 
\[
 \rmh(\theta)= \frac 1 {[\Q(\theta):\Q]} \log \rmM(f) \leqslant \log \rmM(f).
 \]
\end{proof}

We consider a family $\mathcal F = \mathcal F_{\boldsymbol {\mathcal D}}$ of binary fewnomials attached to the data $\boldsymbol {\mathcal D}=(r, (\mathcal E_k))$ as in Definition \ref{fewnomial}. Let $(a_0,a_1,\dots,a_r)\in\calE_k$ and let
\[
h(t)=
a_0t^r+a_1t^{r-1}+\cdots+a_r\in\Z[t]
\]
be the polynomial associated with $(a_0,a_1,\dots,a_r)$.
We decompose $h$ into irreducible factors in $\C[t]$:
\[
h(t)=a_0\prod_{j=1}^r(t-\theta_j).
\]
By hypothesis, $\theta_1,\dots,\theta_r$ are pairwise distinct. The degree $D$ of the number field $\Q(\theta_1,\dots,\theta_r)$ 
satisfies $1\leqslant D\leqslant r!$. 

Let $k\geqslant 1$ and $r\geqslant 1$ be two integers such that the product $d=kr$ is $\geqslant 3$. Let 
 \begin{equation}\label{Equation:F}
 F(X,Y)=a_0X^{kr}+a_1X^{k(r-1)}Y^k+\cdots+a_{r-1}X^kY^{k(r-1)}+a_rY^{kr}
 \end{equation}
be the binary form in $\calF_d$ given by \eqref{278bis}.

Recall the Definition \eqref{Equation:calA} of $\calA(F)$. Thanks to Lemma \ref{Lemme:hVersusH}, we have
\[
\max_{1\leqslant j\leqslant r} \rme^{ \rmh(\theta_j)}\leqslant \sqrt{r+1} {\calA}(F).
\] 
The next result follows from Proposition \ref{Prop:9.22}. 

\begin{corollary}\label{Cor:H(xd,yd)}
Let $x$ and $y$ be in $\Z$. Set $\calX:=\max\{|x|,|y|\}$. Let $F$ be as \eqref{Equation:F}. 
Assume 
 $\calX\geqslant 2$ and
 $F(x,y)\not=0$. 
Then 
 \[ 
 |F(x,y)| 
 \ge
\max\{ |a_0x^{d}|, |a_r y^{d}| \} 
 \exp\big\{-C r^{4r} (\log d)(\log \calX)(\log \calA^\star(F))\big\}
 \]
 with the constant $C$ of Proposition \ref{Prop:9.22}.
\end{corollary}

\begin{proof} 
The case $r=1$, $k\geqslant 3$ is \cite[Corollary 5.2]{FW2}.

\ 
When $k=1$ and $d=r\geqslant 3$, the trivial lower bound $|F(x,y)|\geqslant 1$ gives a stronger result, since 
\[
\max\{ |a_0x^{d}|, |a_r y^{d}| \}\leqslant \calA^\star(F) \calX^d. 
\]
We now assume $k\geqslant 2$ and $r\geqslant 2$. 
We may also assume $xy\not=0$ since the result is trivial when $xy=0$.

Let us write 
\[
F(x,y)=a_0 \prod_{j=1}^r(x^k-\theta_jy^k).
\]
By symmetry, since $a_0a_r\not=0$ and since $\rmh(1/\theta_j)=\rmh(\theta_j)$, we may assume $ |a_0x^{d}| \geqslant |a_r y^{d}|$. 
Let $j\in\{1,2,\dots,r\}$. We first use Proposition \ref{Prop:9.22} with 
\[
D=[\Q(\theta_j):\Q]\leqslant r!,\quad 
b_1=k,\quad b_2=1,\quad \alpha_1=\frac y x, \quad \alpha_2=\theta_j,
\]
\[
\log B= \frac{\log k}{\log 2}, 
 \quad
 \log A_1= \frac{\log\calX}{\log 2}, 
 \quad \log A_2= r\log\calA^\star(F).
\]
Using Stirling's formula \cite{Ro} 
\[
r!\leqslant r^r e^{-r} \sqrt{2\pi r} \, \rme^{1/12r}
\]
together with the upper bound 
\[
4\pi^2 r^5 \frac{\log r}{(\log 2)^2} \rme^{1/3r}\leqslant \rme^{4r}
\]
for $r\geqslant 2$ we get
\[
\frac 1 {(\log 2)^2} r!^4r\log(r!)\leqslant r^{4r-1}.
\]
From Proposition \ref{Prop:9.22}, we deduce
\[
\left| 
\theta_j\left ( \frac y x\right)^k -1\right| \geqslant 
 \exp\big\{-C r^{4r-1} (\log k)(\log \calX)(\log \calA^\star(F))\big\}
 \]
for $1\leqslant j\leqslant d$. Hence 
\[
\prod_{j=1}^r\left| 
\theta_j\left( \frac y x \right)^k -1\right| 
\geqslant 
 \exp\big\{-C r^{4r} (\log k)(\log \calX)(\log \calA^\star(F))\big\}
\]
and
\[
\begin{aligned}
|F(x,y)|
&=
|a_0x^{d} | \prod_{j=1}^r\left| 
\theta_j\left( \frac y x \right)^k -1\right| 
\\
& \ge
 |a_0x^{d}| \exp\big\{-C r^{4r} (\log k) (\log \calX)(\log \calA^\star(F))\big\}.
 \end{aligned}
\]
\end{proof}

From Corollary \ref{Cor:H(xd,yd)} we deduce the lower bound 
 \[
 |F(x,y)|\geqslant \calX^{d} \exp\big\{- C r^{4r} (\log d) (\log \calX) (\log \calA^\star(F))\big\}
 \] 
which we write as 
\begin{equation}\label{Equation:diophantienne}
 |F(x,y)|\geqslant \calX^{d - C r^{4r} (\log d) (\log \calA^\star(F))}.
\end{equation}

 \subsection{Asymptotic estimate}
 
 The next result gives an asymptotic upper bound for the number of integers which are represented by binary forms of large degree in the family $\calF$ of binary fewnomials introduced in section \ref{S:Fewnomial}. It also gives an upper bound for the number of representations of such an integer. 
 
 Recall the constant $C=2^{79}3^{15}$ from Proposition \ref{Prop:9.22}.

 \begin{theorem}\label{Th:majorationasymptotique} 
Under the assumptions of Theorem \ref{Th:principal},
let $\lambda$ and $\mu$ be two real numbers satisfying $\lambda>2$ and
\begin{equation}\label{Equation:mu}
 0< \mu<\frac{\lambda-2}{Cr^{4r}\lambda}\cdotp
\end{equation}
Let $d_0\geqslant 3$ be an integer. Assume that the condition 
 \begin{equation}\label{Equation:majorationCalAmu}
\calA^\star(F)\leqslant \exp(\mu d/\log d)
\end{equation}
 is satisfied for all $d\geqslant d_0$ and all $F\in\calF_d$. Then 
 
 \vskip .3cm
 \noindent
(a) For all $m\in\Z\smallsetminus\{-1,0,1\}$ and all $d\geqslant 3$ multiple of $r$, the set $\calG_{\geqslant d}(m)$ is finite. 
Furthermore, for all $(\lambda, \mu,r,d)$ as above with $d\geqslant d_0$, there exists a constant $c_1$, 
only depending on $( \lambda, \mu, r,d)$, such that, for every $\vert m\vert \geqslant 2$, one has the inequality
\[
\sharp \mathcal G_{\geqslant d} (m) \leqslant c_1 \vert m\vert^{\lambda /(2d)}.
\]
 
 \vskip .3cm
 \noindent (b) For all $(\lambda, \mu,r,d)$ as above with $d\geqslant d_0$, there exists $c_2$ depending only on $(\lambda, \mu, r, d)$, such that, for all $ N \geqslant 2$, one has the inequality 
 \[
 \sharp \mathcal R_{\geqslant d} (N) \leqslant c_2N^{\lambda /d}.
 \]
 \end{theorem}

\begin{proof} Define $\lambda'$ by the equality 
\[
\mu = \frac{\lambda' -2} {C r^{4r} \lambda'}\cdotp
\]
By \eqref{Equation:mu}, we have the equalities $2 < \lambda' < \lambda$. The number $\theta := \lambda'/2$ satisfies
\[
\theta >1 \text{ and } 1-\frac 1 \theta = \frac{\lambda'-2}{\lambda'} = Cr^{4r}\mu.
\]
Let $d\geqslant 3$ be a multiple of $r$, written as $d=kr$. Let $m\in \Z$, $\vert m \vert \geqslant 2$ be such that $\mathcal G_{\geqslant d} (m) \neq \emptyset$: there exists $(x,y, F) \in {\mathcal G}_{\geqslant d} (m)$ such that $m=F(x,y)$, where $F\in \mathcal F_{k'r} $ (with $k'\geqslant k$) is the binary form
\[
a_0 X^{k'r} +a_1 X^{k'(r-1)} Y^{k'}+\cdots +a_{r-1} X^{k'} Y^{k'(r-1)} + a_r Y^{k'r}.
\]
Let $\mathcal X = \max \{ \vert x \vert, \, \vert y \vert \}$. By hypothesis, we have $\mathcal X \geqslant 2$. The lower bound \eqref{Equation:diophantienne} yields 
\[
\vert m \vert \geqslant \mathcal X^{k'r -Cr^{4r} (\log (k'r))(\log \mathcal A^\star (F))}.
\]
Assume now $d\geqslant d_0$. 
From \eqref{Equation:majorationCalAmu} we deduce the upper bound
\[
Cr^{4r} (\log (k'r)) (\log \mathcal A^\star (F)) \leqslant Cr^{4r+1} \mu k' = k'r \left( 1 -\frac 1 \theta \right),
\]
hence
\[
\mathcal X^{k'r} \leqslant \vert m \vert^\theta.
\]
Define
\[
M_0:= \left\lfloor \frac{\theta \log \vert m \vert}{r \log 2}\right\rfloor.
 \] 
Thanks to the inequalities $\mathcal X \geqslant 2$ and $k'\geqslant k$ we deduce $M_0\geqslant k$ and
\begin{equation}\label{Equation:M0}
k'\leqslant M_0 \qquad \text{ and } \qquad \mathcal X \leqslant \vert m \vert ^{\theta/ (k'r)}.
\end{equation}
Given $x$ and $m$, the number of $y^{k'}$ such that $F(x,y)=m$ is at most $r$, hence the number of such $y$ is at most $2r$. This shows that the set $\mathcal G_{\geqslant d} (m)$ is finite for all $m$ with $|m|\geqslant 2$ and $d\geqslant d_0$. Since, for all $d \geqslant 3$ all the sets $\mathcal F_d$ are finite, we deduce from Thue's Theorem on the finiteness of the number of solutions of Thue's equation that the set $\mathcal G_{\geqslant d} (m)$
is finite for any $d\geqslant 3$.

We also deduce for $d\geqslant d_0$ that the number of elements in $\mathcal G_{\geqslant d} (m)$ is bounded by
\begin{equation}\label{Equation:sharpGd(m)}
\sharp \, \mathcal G_{\geqslant d} (m) \leqslant 4r \vert m \vert^{\theta/d} \sum_{k'=k}^{M_0} \sharp \mathcal E_{k'}. 
\end{equation}
Using the inequality \eqref{Equation:New1} under the form 
\[
\sharp \mathcal E_{k'} \leqslant 3^{r+1} \max_{F\in \mathcal F_{k'r}} (\mathcal A^\star(F))^{r+1},
\]
together with the hypothesis (2.3) we deduce 
\[
\sharp \mathcal E_{k'} \leqslant 3^{r+1}\, \exp \left(\mu k' r(r+1) /(\log (k'r))\right). 
\]
Combining with \eqref{Equation:sharpGd(m)} and the upper bound $\theta<\lambda/2$, we deduce 
the inequality
\begin{align*}
\sharp \, \mathcal G_{\geqslant d} (m) &\leqslant 4r \vert m \vert^{\theta/d} \cdot 
M_0 \cdot 3^{r+1}\, \exp \left(\mu M_0 r(r+1) /(\log (M_0r))\right) \\
& \leqslant c_1 \vert m\vert^{\lambda/(2d)}.
\end{align*} 
This completes the proof of the item (a).

The proof of the item (b) has similarities and works as follows. Write $d=kr$ and let $m$ be an element of $\mathcal R_{\geqslant d}(N)$.
It satisfies $\vert m \vert \leqslant N$ and it can be written as 
$m= F(x,y)$ for some $(x,y)$ such that $ \mathcal X \geqslant 2$, for some $k'\geqslant k$ and some $F\in \mathcal F_{k'r}$.
However the inequalities \eqref{Equation:M0} imply
\[
k'\leqslant M_1\quad \text{ and } \quad \max\{|x|,|y|\} \leqslant N^{\theta/(kr)}
\quad \text{ where } \quad 
M_1:=
\left\lfloor
 \frac{\theta \log N}{r\log 2} \right\rfloor.
\]
Thus we have
\[
\sharp \mathcal R_{\geqslant d} (N)\leqslant \left(1+2 N^{\theta/(kr)}\right)^2 \sum_{k'=k}^{M_1} \sharp \mathcal E_{k'},
\]
and the item (b) follows.
\end{proof}

 \subsection{Proof of Theorem \ref{Th:principal}}\label{Subsection:ProofPrincipal}
 
 Let $\epsilon >0$ be fixed, and $\lambda = 2+2\epsilon$. Let 
\[
\mu_0 := \frac{ \lambda -2}{C r^{4r} \lambda},
\]
and as mentioned above let 
\[
\eta := \epsilon (2Cr^{4r})^{-1}.
\]
The inequality $\eta < \mu_0$ allows to apply Theorem \ref{Th:majorationasymptotique} (a). Since we have $\lambda /(2d) < (1/d)+\epsilon$
we obtain the upper bound of $\mathcal G_{\geqslant d} (m)$ for $d\geqslant d_0$ as claimed in Theorem \ref{Th:principal} (a).
This completes the proof of Theorem \ref{Th:principal} (a).

We now prove the alinea (b) of Theorem \ref{Th:principal}. We separate its proof according to the size of $d$, 
compared with $d_0$ starting from which, the upper bound \eqref{Equation:majorationCalAeta} is true. 

\vskip .2cm
--- Assume $d \geqslant d_0$. Let us check condition (ii) in the Definition \ref{def:Regular}. Let $\epsilon_1>0$. For $d'>\lambda/\epsilon_1$
Theorem \ref{Th:majorationasymptotique} (b) yields 
 \[
 \sharp \mathcal R_{\geqslant d'} (N) \leqslant c_2N^{\lambda /d'}<N^{\epsilon_1}
 \]
 for sufficiently large $N$. Hence, applying Theorem 2.2 above, with $A=2$, (or \cite[Theorem 1.11]{FW2})
we obtain the alinea (b) in that case.

 \vskip.2cm
--- Assume $3\leqslant d < d_0$ we extend the above proof as follows to take into account the contribution of the forms of $\mathcal F$ with degree in the interval $[d, d_0-1]$. 
We start from the equality
\begin{equation}\label{Equation:18}
\sharp \mathcal R_{\geqslant d} (N) = X+O(Y),
\end{equation}
with
\begin{multline*}
X= \sharp \{ m : \vert m\vert \leqslant N,\, m =F(x,y) \text{ for some } (x,y,F) \\ \text{ with } \max \{ \vert x \vert, \vert y \vert \} \geqslant 2, \, F \in \mathcal F \text{ with } d\leqslant \deg F < d_0\},
\end{multline*}
and 
\[
Y =\sharp \mathcal R_{\geqslant d_0} (N).
\]
Let $d_1= (d_0-1)^\dag$ We have $ d_1\geqslant d_0$ and $d_1\geqslant d^\dag$. We trivially have $Y= \sharp \mathcal R_{\geqslant d_1}(N)$. By the above discussion, we have 
\begin{equation}\label{Equation:19}
Y = O( N^{2/d_1})= O( N^{2/d^\dag}).
\end{equation}
 To deal with $X$, we benefit from the fact that there are finitely many forms in the union $\bigcup_{d\leqslant d' < d_0}\mathcal F_{d'}$. Let $d_2= (d-1)^\dag$. If $d_2 >d$ the leading coefficient on the right--hand side of \eqref{Equation:calRdN} vanishes. We suppose that 
 $d_2 \leqslant d_0-1$ otherwise there is nothing to prove. Let $F\in \mathcal F_{d_2}$. Then \cite[Theorem 1.1]{SX} gives an asymptotic formula for 
\begin{equation}\label{Equation:20}
\sharp \{ m : \vert m \vert \leqslant N, \, m=F(x,y) \text{ for some } (x,y) \text{ with } \max \{ \vert x \vert, \vert y \vert \} \geqslant 2\}. 
\end{equation}
(also see \eqref{Equation:CFN2/d} above). 
If $G \in \mathcal F_{d_2}$, with $G\neq F$ (so $G$ is not $\Q$--isomorphic to $F$ by hypothesis), then \cite[Theorem 1.1]{FW2} gives an upper bound for
\begin{equation}\label{Equation:21}
\sharp \{ m : \vert m \vert \leqslant N, \, m=F(x,y) =G(u,v)\text{ for some } (x,y,u,v) \text{ with } \max \{ \vert x \vert, \vert y \vert \} \geqslant 2\}. \end{equation}
We then apply the inclusion--exclusion principle to give an asymptotic formula for 
\[
\sharp \{ m : \vert m \vert \leqslant N, \, m=F(x,y) \text{ for some } (x,y) \text{ with } \max \{ \vert x \vert, \vert y \vert \} \geqslant 2 \text{ and some } F \in \mathcal F_{d_2}\}.
\]
Since each set $\mathcal F_d$ is finite, we deduce from \eqref{Equation:CFN2/d} the bound 
\begin{equation}\label{Equation:22}
\begin{aligned}
\sharp \{ m : \vert m \vert \leqslant N, \, m=F(x,y) \text{ for some } (x,y) \text{ with } \max \{ \vert x \vert, \vert y \vert \} \geqslant 2\\ \text{ and some } F \in \bigcup_{d_2^\dag \leqslant \ell < d_0}\mathcal F_{\ell}\} =O(N^{2/d_2^\dag}). 
\end{aligned}
\end{equation}
By \eqref{Equation:18}, \eqref{Equation:19}, \eqref{Equation:20}, \eqref{Equation:21} and \eqref{Equation:22} we complete the proof of alinea (b) of Theorem \ref{Th:principal}.
 
 The item (c) of Theorem \ref{Th:principal} directly follows from the item (b) by a combination of the definition \eqref{defCF}, the equality \eqref{=1/2ou1/4} and the definition of $A^\star_F$.
 
 \section{Examples of sets of $\C$ and $\Q$--homography--free sets}\label{Section:Examples}
 
 Our next task is to exhibit examples of families of binary fewnomials suited for an application of Theorem \ref{Th:principal} (c).
Recall that $\K = \Z,\, \Q, \R$ or $\C$. 
\subsection{Old examples}

Such examples were already given in previous papers of the authors through the following natural approach. Let $F$ and $G$ be two binary forms of $\Bin (d, \K)$ and suppose that we are interested in the 
$\gamma \in {\GL}(2, \K)$ such that $F\circ \gamma =G$. The study of these $\gamma$ is essentially equivalent to the study of the homographies $\mathfrak h$ with coefficients in $\K$ which exchange the complex roots $\rho$ of the polynomials
$f(t):= F(t,1)$ and $g(z):= G(z,1)$ (see Lemmas \ref{278} and \ref{iso<->auto} and Proposition \ref{Prop3.9} below). When $\K =\Q$, we studied the following cases 

 \vskip .2cm
\begin{itemize}
\item $F$ and $G$ are cyclotomic forms with the same degree, then the $\rho$ are primitive roots of unity (see \cite[Proposition 4.8 and Corrigendum]{FW0}), 

 \vskip .2cm
\item $F$ and $G$ are products of distinct irreducible quadratic forms of the shape $X^2+\alpha Y^2$ with $\alpha\in\Z$, then the roots $\rho$
are algebraic irrational numbers 
with degree two, so necessarily we have $\mathfrak h (\Q (\sqrt {-\alpha})) = \Q (\sqrt{- \alpha})$ (see \cite[Propositions 4.1 \& 5.1]{FW1} and \cite[Proposition 3.1]{FW2}), 

 \vskip .2cm
\item $F$ and $G$ are products of distinct linear factors of the shape $X-aY$, with $a\in \Q$, then we exploit the fact that $\mathfrak h$ preserves the cross ratios of any $4$--tuples of distinct $\rho$ (see \cite[Proposition 6.1]{FW1}).
\end{itemize}
Apart from products of binomial forms $(X^r+\alpha Y^r)$ (for suitable rational integers $\alpha$), the above examples do not seem to lead to interesting examples of binary fewnomials.

\subsection{New examples} 

The landscape of Theorem \ref{Th:principal} is different since 
the information concerning the binary fewnomials is not of algebraic nature but it concerns the indices where the corresponding 
coefficients of the form vanish.
 Theorems \ref{486} and \ref{527} below are 
written in that sense. The proofs of these results are based on the above homographies $\mathfrak h$ and on the symmetric functions of the roots of a polynomial. 
The number and the indices of these zero coefficients are important. However, one can prove variations of our results by shifting the string of these zeroes. 
We will not investigate these possible extensions.

To state our result, we introduce the following conventions: Let $F(X,Y)$ be a binary form, not necessarily a binary fewnomial, written as
\begin{equation}\label{defF}
F(X,Y) =a_0 X^d+a_1 X^{d-1} Y + \cdots +a_d Y^d.
\end{equation}
We suppose that $a_0a_d \neq 0$. We define the two functions
\begin{equation}\label{defLambda}
\begin{cases}
\Lambda^+ (F) := \max\{ \ell : a_i= 0, 0 < i <\ell \},\\
\Lambda^- (F) := \min \{\ell : a_i=0, \ell < i < d\}.
\end{cases}
\end{equation}
They satisfy the properties 
\[
a_{\Lambda^\pm (F)} \neq 0,\ 1\leqslant \Lambda^+ (F) \leqslant \Lambda^- (F) \leqslant d 
\] 
and
\[
\Lambda^\mp (F) = \Lambda^\pm (F^{\rm rec}),
\] 
where 
$F^{\rm rec} (X,Y)$ is the {\em reciprocal binary form} defined by $F^{\rm rec} (X,Y) := F(Y,X)$. So 
we restrict ourselves to statements (for instance Theorem \ref{486} or \ref{527}) in terms of $\Lambda^+ (F)$ only.

The first result (Theorem \ref{486}) considers the case of binary forms where, for the very first positive values of $i$, the coefficients $a_i$ are equal to zero. The following definition is essential. 
\begin{definition}\label{Def:reduced}[Reduced set of binary forms]
 Let $\K$ be as above and let $d\geqslant 3$ be an integer. A set $\mathcal E$ of binary forms of degree $d$ with coefficients in $\K$ 
 is called {\em $\K$--reduced} 
 if it satisfies the three conditions 
\begin{enumerate}
\item for any $F$ in $\mathcal E $, we have $a_0a_d \not= 0$, 
\item the set $\mathcal E $ is $\K$--dilation--free, 
\item there is no pair $(F,G)$ of distinct binary binomial forms \eqref{Equation:binomialForms} of $\mathcal E$ and no pair $(u,v) \in (\mathbb C^\times)^2$ such that $F(vY,uX) =G(X,Y)$. 
\end{enumerate}
\end{definition}
If the set $\mathcal E$ is $\K$--reduced 
so does every subset of $\mathcal E$.
The item 3 is satisfied when $\mathcal E$ contains one binomial form at most. 
When $\K = \C$, the condition 3 is satisfied if and only if ${\mathcal E}$ contains at most one binomial form. 

We will prove in \S \ref{1033}:
\begin{theorem}\label{486}
 Let $d\geqslant 3$ and $\K$ as above. Let $\mathcal E$ be a $\K$--reduced set of binary forms. Assume that any $F\in\mathcal E$ satisfies
 \[ 
 \Lambda^+ (F) \geqslant \frac{d+3} 2 \cdotp
 \]
 Then the set $\mathcal E$ is $\K$--homography--free.
\end{theorem}

The assumption $ \Lambda^+ (F) \geqslant (d+3)/2$ implies that when $d\in\{3,4\}$, the set $\mathcal E$ only contains binomial forms.

Theorem \ref{486} is quite general and the discussion in \S \ref{discussion} will show that the lower bound
$\Lambda^+ (F) \geqslant (d+3)/2$ is optimal.


A different way to see the quasi--optimality of Theorem \ref{486} is the following Proposition, where we follow the notations \eqref{defF} and where we use basic concepts of linear algebra. We introduce the subset of $8$ matrices in $\GL(2,\K)$: 
\[ 
\mathcal G=
\left\{ \begin{pmatrix} \pm 1&0\\0&\pm 1\end{pmatrix}, \, \begin{pmatrix} 0&\pm 1\\\pm 1&0\end{pmatrix}\right\}  
\]
 We have
\begin{proposition} For every $d\geqslant 3$ there exists a binary form $F$ with degree $d$ and with integer coefficients, such that
\begin{itemize}
\item $a_1=a_2=\cdots a_{\lfloor d/2\rfloor}=0,$
\item there exists $\gamma\in {\rm GL}(2,\Z)\smallsetminus \mathcal G$, such that $F\circ \gamma =F$.
\end{itemize}

\end{proposition}
\begin{proof}  
Let $S$ be the following matrix 
\[
S = \begin{pmatrix}
0&1\\
1&0\end{pmatrix}
\]
which is attached to the change of variables $(X,Y) \mapsto (Y,X)$.
 Let $\mathcal E(d, \Q)$ be the $\Q$--vector space 
gathering all the binary forms with degree $d$, rational coefficients together with the $0$--form.
It has dimension $d+1$. If $\xi $ belongs to ${\rm GL} (2, \Z)$ we denote by $\xi^\dag $ the automorphism of $\mathcal E(d, \Q)$ defined by
\[
\xi^\dag (F) = F \circ \xi^{-1} \ (F \in \mathcal E(d, \Q)).
\]
In particular we have the equality $(\xi \eta)^\dag = \xi^\dag \eta^\dag,$ for any $\xi$ and $\eta \in {\rm GL}(2,\Z)$.
Of particular importance is the vector subspace 
\[
\mathcal S(d,\Q) :=\{ F \in \mathcal E(d,\Q) : a_k = a_{d-k} \, (0\leqslant 2k\leqslant d)\}.
\]
It is the eigenspace (relative to the eigenvalue $1$) of the automorphism $S^\dag $ of $\mathcal E(d,\Q)$. A basis of $ \mathcal S(d,\Q) $ is given by 
\[
\{ X^kY^{d-k} + X^{d-k}Y^k: 0\leqslant 2k \leqslant d\}.
\]
We have the equality $\dim \mathcal S(d, \Q)= \lfloor d/2\rfloor+1$. Let $\xi$ be any element of ${\rm SL}(2,\Z)$ such that the element 
\[
\gamma :=\xi S \xi^{-1}
\]
does not belong to $\mathcal G$. Then $\xi^{\dag} \left( \mathcal S(d,\Q) \right) $ is the eigenspace of $\gamma^\dag$ relative to the eigenvalue $1$. It also has dimension $\lfloor d/2\rfloor +1$.

For $1\leqslant \kappa <d$, let $\mathcal V(d, \kappa, \Q)$ be the vector subspace
\[
\mathcal V(d, \kappa, \Q):=\{ F:a_1=\cdots=a_\kappa =0 \}.
\]
Its codimension is equal to $\kappa$. If one has the inequality
\[
(d+1-\kappa) + ( \lfloor d/2\rfloor +1)> d+1,
\]
which is equivalent to the inequality
\[
\kappa \leqslant \lfloor d/2\rfloor,
\]
then the intersection of vector spaces $\mathcal V(d, \kappa, \Q) \cap \xi^\dag \left(\mathcal S(d,\Q)\right) $ contains an element $F$ different from $0$. Obviously, for this $F$, we have $a_1 = \cdots = a_\kappa=0$ and $ \gamma^\dag (F) = F\circ \gamma^{-1} = F$.
\\
Note that this construction is too general to ensure that the form $F$ has a discriminant $\neq 0$ and a first coefficient $a_0\neq 0$
\end{proof}
\subsubsection{Three illustrations.} 
$\bullet $ We now consider $d=3$ (so $\lfloor d/2\rfloor =1$) with the choices 
\[
\xi =\begin{pmatrix} 2& 1\\3& 1\end{pmatrix}, \xi^{-1} =\begin{pmatrix} -1& 1\\3& -2\end{pmatrix} \text{ and } \gamma = \xi S \xi^{-1} =\gamma^{-1}= \begin{pmatrix}5& -3\\ 8& -5\end{pmatrix}.
\]
Let $\phi (X,Y)$ be the cubic symmetric form
\[
\phi (X,Y):= 13(X^3+Y^3) +51 (X^2 Y +XY^2).
\]
Then we have
\[ F(X,Y) = 
\left(\xi^\dag (\phi)\right)(X,Y) = (\phi \circ \xi^{-1}) (X,Y) = 32 X^3 -30XY^2 +11 Y^3.
\]
We check the equality
$ \gamma^\dag (F) = F\circ \gamma^{-1} =F,$ since we have the equality
\begin{equation}\label{Equation:F(5X-3Y,8X-5Y)}
F(X,Y) = F(5X-3Y, 8X-5Y).
\end{equation}

\noindent
$\bullet$ The same equality \eqref{Equation:F(5X-3Y,8X-5Y)} holds for 
\[
F(X,Y)=256 X^4 -240 XY^3 +111Y^4
\]
using the quartic symmetric form 
\[
\phi (X,Y)=127 (X^4+Y^4) + 740 (X^3 Y+ X Y^3) + 1338 X^2 Y^2.
\]

\noindent
$\bullet$ We consider $d=10$ (so $\lfloor d/2\rfloor =5$) and we require the help of a computer. Consider the symmetrical form $\phi (X,Y)$ defined by
\[
\begin{gathered}
\phi (X, Y) =76 \, 210 \, 176 \, 793 \, \left (X^{10}+Y^{10}\right)
+872 \, 977 \, 899 \, 590 \, \left( X^9Y+XY^9\right)\\
+ 4 \, 381 \, 399 \, 953 \,765 \left( X^8Y^2+X^2 Y^8\right) 
+ 12\, 658\, 497\, 992\, 520\left( X^7Y^3 + X^3 Y^7\right)
\\
+23 \, 266 \, 629 \, 555 \, 330\left(X^6Y^4+X^4 Y^6\right) + 28\, 385\, 698\, 168\, 548\, X^5Y^5.
\end{gathered}
\]
Let 
\[
\xi = \begin{pmatrix} 1&2\\3&5 \end{pmatrix}, \xi^{-1} = \begin{pmatrix}-5 & 2 \\ 3&-1\end{pmatrix} \text{ and } \gamma = \xi S \xi^{-1} =\gamma^{-1}= \begin{pmatrix}-7& 3\\ -16& 7\end{pmatrix}
\]
 By a ${\rm GL}(2, \Z)$ change of variables, we have
\[
\begin{gathered}
F(X, Y) =\left(\xi^\dag(\phi) \right)(X,Y)=(\phi \circ \xi^{-1})(X,Y)=\phi (-5X+2Y, 3X-Y)\\
=-34 \, 359 \, 738 \, 368 \, X^{10}
+49 \, 565 \, 859 \, 840 \, X^4Y^6- 74 \, 095 \, 902 \, 720 \, X^3Y^7 
\\
+42 \, 402 \, 890 \, 880 \, X^2Y^8
- 10 \, 956 \, 131 \, 760 \, XY^9 + 1 \, 074 \, 852 \, 609 \, Y^{10}.
\end{gathered}
\]
We check the equality
$ \gamma^\dag (F) = F\circ \gamma^{-1} =F,$ since we have
\[
F(X,Y) = F(-7X+3Y, -16X+7Y).
\]
With the help of a computer, one checks that the form $F$ has a discriminant different from $0$ and that the coefficients are relatively prime.
 
 \subsubsection{A variation of Theorem \ref{486}} 
 
Theorem \ref{486} has the defect not to cover the case where $\Lambda^+ (F)$ is very close from $d/2$
(i.e, $2\Lambda^+(F) = d$, $d+1$ or $d+2$). For instance it does not apply to the balanced trinomial forms appearing in \eqref{balancedtrinomial}. In order to circumvent this failure, we introduce the following sets, where we impose to the forms to have
a string of (at least) four zero monomials located just after the monomial $a_{\Lambda^+ (F)} X^{d-\Lambda^+ (F)} Y^{\Lambda^+ (F)}$.
 
 We will prove in \S \ref{Proof527et2} the following analogue of Theorem \ref{486}.

\begin{theorem} \label{527} Let $d\geqslant 11$ and $\K$ as above. Let $\mathcal E$ be a $\K$--reduced subset of $\Bin (d, \K)$.
Assume that for each $F\in E$ we have 
\[
d/2\leqslant \Lambda^+ (F)\leqslant d-4, \; a_k =0 \text{ for } \Lambda^+ (F)+1 \leqslant k \leqslant \Lambda^+ (F) +4.
\]
and that if $d$ is even, then $\mathcal E $ contains no trinomial of the form
\begin{equation}\label{Equation:trinome}
a_0 X^d + a_{d/2} X^{d/2} Y^{d/2} + a_d Y^d. 
\end{equation}
Then $\mathcal E$ is $\K$--homography--free.
\end{theorem}
 
 The assumption $ \Lambda^+ (F)\leqslant d-4$ implies that the set $\mathcal E $ contains no binomial form. 
 
 The set $\mathcal E$ does not contain two forms $F$ and $G$ such that $F(vY,uX)=G(X,Y)$ with $u$ and $v$ in $\C^\times$: indeed, if two forms $F$ and $G$ related by such an equation satisfy $\Lambda^+ (F)\geqslant d/2$ and $\Lambda^+ (G)\geqslant d/2$ and are not binomial forms, then $d$ is even and $F$ and $G$ are trinomials of the form \eqref{Equation:trinome}.
 
 The assumption $d\geqslant 11$ is necessary in view of the other conditions of Theorem \ref{527}. For $d = 11$ the elements of $\mathcal E$ are of the form
\[
a_0X^{11}+a_6X^5Y^6+a_{11}Y^{11}
\]
with $a_0a_6a_{11}\not=0$. 
 
\subsection{Examples of homography--free sets of binary forms}\label{SS:HomographyFreeSets}

We give examples where the assumptions of Theorem \ref{Th:principal} (b) and (c) are satisfied. 

\subsubsection{Corollaries to Theorem \ref{486}} 

These corollaries concern two sets of binary forms.
The first one is:
\begin{equation}\label{defUd1}
\mathcal U^{(1)}_{d}(\K):= 
\left\{ F\in \Bin (d, \K) : a_0\neq 0, \; \Lambda^+ (F) \geqslant (d+3)/2, \; a_{d-1}= a_d =1 \right\}.
\end{equation}
 The condition $a_{d-1}\not=0$ implies $\Lambda^+ (F) \leqslant d-1$, hence if $\mathcal U^{(1)}_{d}(\K)$ is not empty then $d\geqslant 5$. Conversely, for $d\geqslant 5$, the set $\mathcal U^{(1)}_{d}(\K)$ is infinite. This is a consequence of the following lemma with $e=1$.

\begin{lemma}\label{Lemma:DiscriminantTrinomes} 
Let $a$, $d$ and $e$ be integers such that $a\not=0$, $1\leqslant e\leqslant d-1$. Then the discriminant of the binary form 
\[
F_a(X,Y):=aX^d+X^eY^{d-e} +Y^d
\]
is different from zero.
\end{lemma} 

The form $F_a$ has $\Lambda^+(F_a)=d-e$ and $d-1\geqslant (d+3)/2$ for $d\geqslant 5$. 

\begin{proof}
We need to prove that the polynomial 
\begin{equation}\label{Equation:392}
f_a (t) = a t^d +t^e+1
\end{equation} 
has no multiple root. Without loss of generality we may assume that $e$ and $d$ are coprime integers. 
Suppose that such a multiple root, that we call $\rho$, exists. It would satisfy $f_a(\rho)=a\rho^d + \rho^e +1= 0$ and $f'_a(\rho)= a d \rho^{d-1} +e\rho^{e-1}=0$, from which we deduce that 
\[
\rho^{d-e}=-\frac e {ad}
\quad\text{ 
and that 
}\quad
\rho^e= - \frac d{{d-e}}\cdotp
\]
However the equation 
\[
a^e=(-1)^d \frac{e^e(d-e)^{d-e}}{d^d}
\]
cannot hold with $a\in\Z$ and $e$ and $d$ coprime integers. 
\end{proof}

To define the second set we recall a classical definition: let $k\geqslant 2$ be an integer and let $x= a/b$ be a non zero rational number, written in its minimal form with $b>0$. We say that $x$ is {\em k--free } if there is no prime $p$ such that $p^k$ divides $ab$.

\noindent 
We introduce the following subset $\mathcal U_d^{(2)} (\Z)$ of $\Bin  (d, \Z)$ containing all the forms $F$ (written as in \eqref{defF}) such that
\begin{equation}\label{defUd2}
\begin{cases}
{\mathrm (a)}\;
a_0>0, \,a_d\not=0, \,
 \, (d+3)/2\leqslant \Lambda^+ (F) \leqslant d-1,\, 
\\
{\mathrm (b)}\;
 a_0 \text{ and } a_d\text{ are } d\text{--free},\\
 \text{(c) if there is an odd index } k \text{ such that } a_k \neq 0\\
 \qquad\text{ then for the smallest such } k \text{ we have } a_k >0.
\end{cases}
\end{equation} 
When $a_0$ is $d$--free, the form $F_{a_0} (X,Y) $ defined in Lemma \ref{Lemma:DiscriminantTrinomes} also belongs to $\mathcal U_d^{(2)} (\Z)$, thus this set is infinite.

We state the following corollary to Theorem \ref{486}. The proof is given in \S \ref{proof2.2}. 
 
\begin{corollary} \label{Corollary2star}
The following properties hold.
\begin{enumerate} 
\item For every $d\geqslant 5$ 
the infinite set $\mathcal U^{(1)}_{d} (\K)$
is a $\K$--homography--free set of binary forms.
\item For $d \geqslant 5$
the infinite set $\mathcal U^{(2)}_d (\Z)$ is a $\Q$--homography--free set of $\Q$--rigid binary forms.
\end{enumerate}
\end{corollary}
 
\subsubsection{Corollaries to Theorem \ref{527}}

We give new examples where the assumptions (b) and (c) of Theorem \ref{Th:principal} are satisfied.
 
Our next example is the set 
\begin{align}\label{defVd1}
\mathcal V^{(1)}_{d}(\K):= 
\bigl\{ F\in \Bin (d, \K) : &a_0\neq 0, \; d/2\leqslant \Lambda^+ (F)\leqslant d-6, \; a_{d-1}= a_d =1
\\
\notag
& 
a_{k}=0\; \text{ for } \; \Lambda^+ (F)+1\leqslant k \leqslant \Lambda^+ (F) +4
 \bigr\}.
\end{align}
 
We introduce the following subset $\mathcal V_d^{(2)} (\Z)$ of $\Bin  (d, \Z)$ containing all the forms $F$ (written as in \eqref{defF}) such that
\begin{equation}\label{defVd2}
\begin{cases}
{\mathrm (a)}\;
a_0>0, \; a_d\not=0, \;, d/2\leqslant \Lambda^+ (F)\leqslant d-5, \; 
 \\
 \qquad
 a_{k}=0 \text{ for } \Lambda^+ (F)+1\leqslant k \leqslant \Lambda^+ (F) +4,
\\
{\mathrm (b)}\;
a_0 \text{ and } a_d \text{ are $d$--free},
\\
{\mathrm (c)}\;
\text{if there is an odd index $k$ such that $a_k\not=0$}, 
\\
\qquad \text{then for the smallest such $k$ we have } a_k>0
 \\
{\mathrm(d)} \;
 \text {if $d$ is even, then $F$ is not a trinomial of the form}
 \\
 \qquad
a_0 X^d + a_{d/2} X^{d/2} Y^{d/2} + a_d Y^d. 
\end{cases}
\end{equation}

It is plain that the set $\mathcal V^{(1)}_{d } (\K)$ (resp. $\mathcal V^{(2)}_{d} (\Z)$)
is empty for $d<12$ (resp. for $d<11$). Let us check that for $d\geqslant 12 $ (resp. $d\geqslant 11 $) this set is infinite.

For $d=11$ the set $\mathcal V^{(2)}_{11} (\Z)$ contains the forms
\[
aX^{11}+X^6Y^5+Y^{11}, \qquad a>0 \quad \text{$d$--free},
\]
 as shown by Lemma \ref{Lemma:DiscriminantTrinomes}. 
 
 For $d\geqslant 12$ and $a\in\Z\smallsetminus\{0\}$, consider the binary forms
\[
G_a(X,Y) = X^d + a X^{d-\nu} Y^\nu + XY^{d-1} +Y^d,
\]
where $\nu=\nu_d=\lfloor (d+1)/2\rfloor$.
The next result shows that for $|a|$ sufficiently large $G_a(X,Y) $ belongs to both $\mathcal V^{(1)}_{d } (\K)$ and $\mathcal V^{(2)}_{d} (\Z)$.

\begin{lemma} 
For any $d\geqslant 12$, there exists a constant $A_d$ such that 
\[
\sharp \{ a \in \Z\smallsetminus\{0\} : G_a \notin \Bin  (d, \Z)\} \leqslant A_d.
\]
\end{lemma}

\begin{proof} 
The discriminant of the polynomial $G_a (X,1)$ is equal to $D(a)$ where $D$ is a polynomial in $\Z[T]$. 
 The discriminant of the polynomial $f_1(x)$ introduced in \eqref{Equation:392} is $D(0)$ when $e=1$.
Since $f_1$ has no multiple roots, we have $D(0)\neq 0$ and the polynomial $D$ has only finitely many roots.
\end{proof} 
 
\begin{remark} 
One can check that $A_{12}=0$, hence the sets $\mathcal V^{(1)}_{12 } (\Z)$ and $\mathcal V^{(2)}_{12 } (\Z)$ contain all the quadrinomials which are of the form 
\[
X^{12}+aX^6Y^6+XY^{11}+Y^{12} \quad (a\in\Z).
\] 
Indeed, the discriminant of the polynomial $T^{12}+aT^6+T+1$ is a polynomial of degree $12$ in $a$, with integer coefficients; using a computer, we check that it has no integral solution. 
 \end{remark}
 
 \medskip
Here is a corollary to Theorem \ref{527}. The proof is given in \S \ref{Proof527et2}.

\begin{corollary}\label{Corollary2} For $d \geqslant 12$, we have 
\begin{enumerate}
\item The infinite set $\mathcal V^{(1)}_{d} (\K)$ 
is a $\K$--homography--free set of binary forms.
\item
The infinite set $\mathcal V^{(2)}_{d} (\Z)$ is a $\Q$--homography--free set of $\Q$--rigid binary forms.
\end{enumerate}
\end{corollary}

\begin{remark}
The conditions concerning $a_{d-1}$ and $a_d$ (in \eqref{defUd1} and in \eqref{defVd1}) and $a_0$, $a_k$ and $a_d$ (in \eqref{defUd2} and in \eqref{defVd2}), 
may appear artificial. They are introduced to eliminate possible homotheties between forms and can be replaced by other types of conditions.
\end{remark}

 \section{Homographies between two binary forms } \label{Section:tools}
 
For the proofs of Theorems \ref{486} (in section \ref{1033}) and \ref{527} (in section \ref{S:Proof527}), in order to check the hypotheses arising from Definition \ref{definition:dilation homograph free}, we need to study the homographies between two binary forms and the automorphisms of a binary form.

\subsection{Zeroes of binary forms}

\subsubsection{Case of one form.}\label{3.1.1}

 In this section $\K$ is the field $\Q$, $\R$ or $\C$. To the element 
\begin{equation}\label{a1a2a3a4}
\gamma = \begin{pmatrix} u_1 & u_2 \\ u_3& u_4\end{pmatrix} \in {\GL}(2, \K)
\end{equation}
we associate the {\em homography } $\tilde \gamma$ of $\mathbb P^1(\K)$ defined by the formula 
\[
\tilde \gamma (x:t) = (u_1x+u_2t : u_3 x+u_4t),
\]
where we denote by $(x : t)$ the generic element of $\mathbb P^1 (\K)$. 
Recall that for $\gamma_1$ and $\gamma_2$ in $ {\GL}(2, \K)$ we have the equivalence 
\begin{equation}\label{iff}
\tilde \gamma_1 =\tilde \gamma_2 \iff \bigl(\text { there exists } \lambda \in \K^\times \text{ such that } \gamma_1= \lambda \gamma_2\ \bigr)
\end{equation}
and the formula
\begin{equation}\label{tildetilde}
\widetilde{\gamma_1\,\gamma_2} = \tilde {\gamma_1} \circ \tilde{\gamma_2}
\end{equation}
 
Let $F$ be a form $\Bin  (d,\C)$ and ${\mathcal Z}(F)$ the set of zeroes of $F$ in $\PP^1 (\C)$. By definition, it is the set of classes $(x:t)$ 
of pairs $(x,t)\in \C^2\smallsetminus \{(0,0)\}$ modulo the homotheties such that $F(x,t) =0$. By assumption the cardinality of ${\mathcal Z} (F)$ is $d$ and for $\gamma\in {\GL}(2, \K)$ we have the conjugation equality 
\begin{equation}\label{Z (fcircgamma)}
\mathcal Z (F\circ \gamma) = \tilde \gamma^{-1} \bigl( \mathcal Z (F)\bigr),
\end{equation} 
which follows from the relations 
\[
(x:t) \in {\mathcal Z} (F\circ \gamma)
\iff
\left( F\circ \gamma\right) (x:t)=0 \iff \tilde \gamma (x:t) \in {\mathcal Z} (F)
\iff (x:t) \in \tilde \gamma^{-1}\left( {\mathcal Z}(F)\right).
\]
In particular for $\gamma \in {\rm Aut} (F,\K)$, we have 
\begin{equation}\label{281}
\mathcal Z (F) =\tilde \gamma \bigl( \mathcal Z (F)).
\end{equation}
Thus for $\gamma\in {\rm Aut} (F,\K)$ the restriction of $\tilde \gamma$ to $\mathcal Z (F)$, denoted by $\tilde \gamma_{\vert \mathcal Z (F)}$, is a bijection of $\mathcal Z (F)$. If $E$ is a subset of $\mathbb P^1 (\C)$, 
let ${\rm Aut}( E, \K)$ be the set of bijections $\phi : E\rightarrow E$ such that there exists $\gamma \in {\GL} (2,\K)$ 
with the property $\tilde \gamma_{\vert E} = \phi$. In particular for $\gamma \in {\GL}(2, \K)$ one has the equality 
\begin{equation*}
{\rm Aut} (\mathcal Z (F\circ \gamma), \K) = (\tilde \gamma_{\vert \mathcal Z(F)})^{-1} {\rm Aut } \left( \mathcal Z (F), \K\right) \tilde \gamma_{\vert \mathcal Z (F)}.
\end{equation*}
 We have 
 
\begin{lemma} \label{278}
Let $F \in \Bin (d, \K)$ ($d\geqslant 3$). The \, $\tilde{}$--map
 which transforms $\gamma\in {\GL}(2, \K)$ in the homography $\tilde \gamma$ on $\PP^1 (\C)$ induces a homomorphism 
 \[
\begin{matrix}
\Psi & : &{\rm Aut} (F, \K) &\longrightarrow & {\rm Aut} (\mathcal Z (F), \K)\\
& & \gamma& \mapsto &\Psi (\gamma) = \tilde \gamma_{\vert \mathcal Z (F)}.
\end{matrix}
\]
We have 
\[
 \ker \Psi =\{ \zeta\, {\rm Id}: \zeta \in \K, \ \zeta^d =1\}.
 \]
 Finally when $\K =\C$ the map $\Psi $ is surjective.

\end{lemma}

\begin{proof} The existence and unicitiy of $\Psi$ follows from \eqref{281}.
As a consequence of \eqref{tildetilde}, for any $\gamma_1$ and $\gamma_2 \in 
{\rm Aut}(F, \K)$, we have $\Psi (\gamma_1\gamma_2 ) =\Psi (\gamma_1)\circ\Psi (\gamma_2)$. The determination of $ \ker \Psi$ comes down to finding the matrices
 $\gamma \in {\rm Aut}( F, \K)$ such that $ \tilde \gamma$ fixes every point of $\mathcal Z (F)$. Since $\mathcal Z (F)$ contains $d$ points and since $d\geqslant 3$, we deduce the equality $\tilde \gamma ={\rm Id}$. By \eqref{iff}, the automorphism $\gamma$ is a homothety of the shape $\zeta\, {\rm Id}$, with $\zeta \in \K^\times$. The equality $F\circ \gamma = \zeta^d F$ restricts the possible values of $\zeta$ by the equality $\zeta^d =1$. 
 
 We now suppose that $\K =\C$ to prove that $\Psi$ is surjective.
So let $F\in \Bin  (d,\C)$ and let $\xi \in {\GL}(2,\C)$ such that $\tilde \xi_{\vert {\mathcal Z}(F)}$ belongs to ${\rm Aut} ({\mathcal Z }(F), \C)$. We want to prove the existence of some $\gamma \in {\rm Aut}(F, \C)$ such that $\tilde \gamma =\xi$. By \eqref{Z (fcircgamma)}, we have the equality ${\mathcal Z} (F) = {\mathcal Z}(F\circ \xi)$.
This implies that the forms $F$ and $F\circ \xi$ are proportional since they have the same zeroes with
the same multiplicities. For some $\alpha \in \C^\times$ we have $F\circ \xi = \alpha F$. It remains to put $\gamma := \lambda \xi$, where $\lambda \in \C $ satisfies $\lambda^{-d }= \alpha$ to obtain the equalities 
$\tilde \gamma = \tilde \xi$ and $F\circ \gamma = F$.
\end{proof}
\begin{remark}
\label{remarque3} 
 When $\K= \Q$ or $\R$, the morphism $\Psi$ in Lemma \ref{278}, is not surjective generally speaking as one sees in the following example. Let $a$ and $b$ be two distinct positive real numbers. Let $ F(X,Y)$ in $\Bin (4, \R)$ defined by 
 \[
F(X,Y) := (X-aY)(X+Y/a)(X-bY)(X+Y/b).
\]
We then have 
\[
\mathcal Z (F)= \bigl\{ (a:1), \, (-1/a:1),\, (b:1),\, (-1/b: 1)
\bigr\}. 
\]
Let $\xi : \PP^1 (\C) \rightarrow \PP^1 (\C)$ be the homography defined by $\xi( z):= -1/z$. We check that $\xi_{\vert \mathcal Z (F)}$ belongs to ${\rm Aut}(\mathcal Z (F), \R)$. Searching for $\gamma \in {\rm Aut}(F, \R)$ such that $\tilde \gamma = \xi$ is equivalent to searching for $\lambda \in \R^\times$ such that 
\[
\gamma =
\begin{pmatrix}
0 &-\lambda\\
\lambda & 0
\end{pmatrix},
\]
satisfies $(F\circ \gamma)\, (X,Y) =F(X,Y)$. Such a $\lambda\in \R$ does not exist, since a direct computation leads to the equality $(F\circ \gamma) (X, Y) =- \lambda^4 \cdot F (X,Y)$.
\end{remark}
Lemma \ref{278} has the following consequence.
\begin{lemma}\label{278ter} 
Let $d\geqslant 4$ and let $F\in \Bin  (d, \Q)$. Suppose that 
${\rm Aut }({\mathcal Z}(F), \C) = \{{\rm Id}\}$. Then we have 
\[ 
{\rm Aut}(F, \Q) 
 = \begin{cases} \left\{ {\rm Id}\right\} &\text{ if } 2 \nmid d \\ \left\{ \pm {\rm Id}\right\} &\text{ if } 2 \mid d.
 \end{cases}
\]
\end{lemma}
\begin{proof} Since ${\rm Aut}(\mathcal Z (F), \C)$ contains one element only, so does ${\rm Aut}({\mathcal Z} (F), \Q)$. This element is ${\rm Id}$, hence one has the equality
\[
{\rm Aut} (F, \Q)= \Psi^{-1} (\{{\rm Id}\})= \ker \Psi,
\]
and the result follows from Lemma \ref{278}.
\end{proof}
Lemma \ref{278ter} never applies when $d=3$, since, in that case, ${\rm Aut}(\mathcal Z (F), \C)$ always has six elements.

\medskip 
We finish this section by some conventions and notations. The results of \S \ref{Section:Examples} concern binary forms $F(X,Y)$ having $a_0\neq 0$ with the notation \eqref{defF}. Thus it is natural to identify the roots of the form $F(X,Y)$ in $\PP^1 (\C)$ with the zeroes (on the complex affine line) of the {\em associated polynomial} 
\begin{equation}\label{669}
f(t)= F(t,1)/a_0.
\end{equation}
By construction, this polynomial is monic and we are led to consider, for $d \geqslant 1$, the following set
of polynomials
\begin{equation}\label{defPd(K)}
\mathcal P_d (\K):= \{ f\in \K [t]: f\text{ monic},\, \deg f =d, {\rm disc } f \neq 0 \},
\end{equation}
with $\K= \C$, $\R$, $\Q$ or $\Z$.
If $f$ and $g$ belong to $\mathcal P_d (\K)$, we systematically write them as 
\begin{equation}\label{670*}
f(t) = t^d +\alpha_1 t^{d-1} + \cdots + \alpha_d,
\end{equation}
and 
\begin{equation}\label{670**}
g(z) = z^d +\beta_1 z^{d-1} + \cdots + \beta_d.
\end{equation}
By convention, we put $\alpha_0 = \beta_0 =1$.

By analogy with \eqref{defLambda}, we define 
\begin{equation}\label{deflambdas}
\Lambda^+ (f) := \max\{ \ell : \alpha_i= 0, 0 < i <\ell \}.
\end{equation}
Since ${\rm disc } f \neq 0$, the polynomial $f(t)$ is not the monomial $t^d$, thus we have 
\begin{equation}\label{1leqleqd}
1\leqslant \Lambda^+ (f)\leqslant d.
\end{equation}
 By analogy with $\mathcal Z (F)$, we introduce the set $\mathcal Z (f):= \{ \rho \in \C : f (\rho) =0\}. $ This set of zeroes has cardinality $d$.

\subsubsection{Case of two forms.} 

This paragraph generalizes the section \ref{3.1.1}
by studying the links between the zeroes of two forms $F_1$ and $F_2$. 

Since we do not wish to consider multiplicities for the zeroes, we need to assume that the discriminants are not zero: the two binary forms $(X-Y) (X-2Y)^2(X-3Y)^2$ and $(X-Y)(X-2Y)(X-3Y)^3$ have the same degree, the same sets of zeroes (not with the same multiplicities), and they are not isomorphic. 

\begin{definition}
Let $\K$ be one of the fields $\Q$, $\R$ or $\C$. Let $\mathcal E_1$ and $\mathcal E_2$ be two subsets of $\PP^1 (\K)$ with equal cardinalities $\geqslant 3$. We call {\em $\K$--isomorphism between $\mathcal E_1$ and $\mathcal E_2$} any bijection $\phi$ from $\mathcal E_1$ on $\mathcal E_2$ such that there exists $h\in {\GL}(2,\K)$ with the property of restriction 
\[
\tilde h_{\vert \mathcal E_1} = \phi. 
\]
The set of these isomorphisms is denoted by ${\rm Isom}\bigl(\, \mathcal E_1, \mathcal E_2; \K\, \bigr)$. If this set is not empty, we say that $\mathcal E_1$ and $\mathcal E_2$ are {\em $\K$--isomorphic}.
\end{definition}
When $\sharp \mathcal E_1 = \sharp \mathcal E_2 =3$, then the sets $\mathcal E_1$ and $\mathcal E_2$ are $K$--isomorphic.
We will use 
\begin{lemma}\label{iso<->auto} Let $d \geqslant 3$. Let $F_1$ and $F_2$ be two forms of $\Bin  (d, \C)$. We suppose that $\mathcal Z (F_1)$ and $\mathcal Z (F_2)$ are $\C$--isomorphic and let $\mathfrak h$ be
an element of ${\rm Isom} (\mathcal Z (F_1), \mathcal Z(F_2); \C)$. Then there exists at least one element $\gamma_0 \in {\GL}(2, \C)$ such that 
\begin{equation} \label{348}
\tilde {\gamma_0}_{\vert {\mathcal Z (F_1)} }= \mathfrak h \text{ and } F_1= F_2 \circ \gamma_0.
\end{equation}
Finally when one such element $\gamma_0$ is fixed, we have the equality
\begin{equation}\label{equalityofsets}
\{ \gamma \in {\GL}(2, \C): \tilde {\gamma }_{\vert {\mathcal Z }(F_1)} = \mathfrak h \text{ and } F_1= F_2 \circ \gamma\} = \{ \lambda\, \gamma_0 : \lambda \in \C, \lambda^d=1 \}.
\end{equation}
\end{lemma}

It follows that $F_1$ and $F_2$ in $\Bin  (d, \C)$ are $\C$--isomorphic if and only if $\mathcal Z (F_1)$ and $\mathcal Z (F_2)$ are $\C$--isomorphic. 

\begin{remark}\label{remarque4}
 When $\mathfrak h$ belongs to ${\rm Isom}(\mathcal Z (F_1), \mathcal Z (F_2); \Q)$, we cannot ensure the existence $\gamma_0\in {\GL}(2, \Q)$ satisfying \eqref{348}. This is the content of Remark \ref{remarque3} above, with the choice $F_1=F_2=F$ and $\mathfrak h = \xi|_{\mathcal Z (F)}$. 
\end{remark}

\begin{proof} By the definition of $\mathfrak h$, there exists $\gamma \in {\GL}(2, \C)$ such that $\tilde \gamma_{\vert \mathcal Z (F_1)} = \mathfrak h$. Write $\gamma $ as in \eqref{a1a2a3a4} 
and 
\[
F_1 (X,Y) =\prod_{i=1}^d \bigl(\, \alpha_i X - \beta_i Y\, \bigr).
\]
So we have $\mathcal Z (F_1)= \bigl\{ (\beta_i: \alpha_i) : 1\leqslant i \leqslant d\bigr\}$ and 
\[
\mathfrak h (\beta_i: \alpha_i)= (u_1\beta_i + u_2 \alpha_i:
u_3 \beta_i + u_4 \alpha_i).
\]
Since $ \mathcal Z (F_2) = \mathfrak h (\mathcal Z (F_1)) $ we deduce the equality
\[
\mathcal Z (F_2) = \bigl\{ (u_1\beta_i + u_2 \alpha_i:
u_3 \beta_i + u_4 \alpha_i) : 1\leqslant i \leqslant d \bigr\},
\]
and the existence of some 
 $c\in \C^\times$ such that
\[
F_2(X,Y) = c \prod_{i=1}^d \left( (u_3 \beta_i + u_4 \alpha_i)X- (u_1\beta_i + u_2 \alpha_i)Y
\right).
\]
From this equality and from the equality 
\[
 ( u_3 \beta_i + u_4 \alpha_i) (u_1X+u_2Y) -(u_1\beta_i + u_2 \alpha_i)(u_3X +u_4Y) = (\det \gamma) \, (\alpha_i X-\beta_iY),
\]
we deduce the equality
\[
F_2 \circ \gamma = c\,(\det \gamma)^d \,F_1.
\]
Define $\gamma_0 := c^{-1/d} (\det \gamma)^{-1}\, \gamma$ where $c^{1/d}$ is any $d$--th root of $c$. 
Then
$\gamma_0$ satisfies \eqref{348}.

To prove \eqref{equalityofsets}, we notice that $\tilde \gamma$ and $\tilde \gamma_0$ co\" incide on a set of $d\geqslant 3$ points of $\PP^1 (\C)$. So they are equal. By \eqref{iff}, there exists 
$\lambda \in \C^\times$ such that $\gamma =\lambda \, \gamma_0$. By hypothesis we have $F_1= F_2 \circ \gamma_0 =F_2 \circ \gamma$.
But $F_2$ is homogeneous with degree $d$ which leads to the condition $\lambda^d =1$. 
\end{proof}

In the same order of ideas as Lemma \ref{iso<->auto} we have 
\begin{proposition}\label{Prop3.9}
Let $d \geqslant 3$ and let $F$ and $G$ two distinct forms of $\Bin  (d, \C)$ such that
\[
{\rm Isom} (\mathcal Z (F), \mathcal Z (G); \C) = \emptyset.
\]
Then we have the equality
 \[
 \{ \gamma \in {\GL}(2, \C) : F= G\circ \gamma \}
 =\emptyset.
 \]
\end{proposition}
\begin{proof} By contraposition, suppose that there exists $\gamma \in {\GL}(2, \C)$ such that $F=G\circ \gamma$.
By \eqref{Z (fcircgamma)}, we have 
$\tilde \gamma^{-1} (\mathcal Z (G)) = \mathcal Z (F)$. Thus $\tilde \gamma$ is an isomorphism between $\mathcal Z (F)$
and $\mathcal Z(G)$.
\end{proof}

 
\subsubsection{From binary forms to polynomials}
We will define the action of the homographies on the set of polynomials in $\mathcal P_d (\K)$.

 Let $F$ and $G$ be two elements in $\Bin(d,\K)$ written as  
 \begin{equation}\label{F=G=}
 \begin{cases} F(X,Y) = a_0 X^d + a_1 X^{d-1} Y + \cdots + a_d Y^d
 \\ 
 G(X,Y) = b_0 X^d+ b_1 X^{d-1} Y+ \cdots + b_d Y^d
 \end{cases}
 \end{equation}
 and let $\gamma$ be an element in $\GL(2,\K)$ written as in \eqref{abcd}. Assume
 $F=G\circ\gamma$:
 \begin{equation}\label{equation:F=Gcircgamma}
 F(X,Y)=G(u_1X+u_2Y,u_3X+u_4Y).
 \end{equation}
Consider the two monic polynomials $f$ and $g$ in $\mathcal P_d (\K)$ associated with the binary forms $F$ and $G$ respectively (recall the definition \eqref{669}):
 \[
 f(t)=\frac 1 {a_0} F(t,1) \text{ and } g(z)=\frac 1 {b_0} G(z,1).
 \]
 The relation  \eqref{equation:F=Gcircgamma} gives 
  \[
 f(t)=\frac {b_0}{a_0} (u_3t+u_4)^dg\left( \frac {u_1t+u_2}{u_3t+u_4}\right)
  \]
and 
\begin{equation}\label{Equation:a_0}
 \frac {a_0}{b_0}= c(\gamma,g)
 \text{ where } c(\gamma,g):=\begin{cases}
 u_3^dg(u_1/u_3)&\text{ if $u_3\not=0$},
 \\
 u_1^d&\text{ if $u_3=0$}.
 \end{cases}
\end{equation}

\begin{definition}\label{Definition:h(f)=g}
 When $f$ and $g$ are two monic polynomials in $\mathcal P_d (\K)$ and $\mathfrak h$ an homography with matrix $\gamma$, we write $\mathfrak h(f)=g$ if $f$ and $g$ are related by the equation
 \[
 f(t)=\frac 1 {c(\gamma,g)}(u_3t+u_4)^dg\left( \frac {u_1t+u_2}{u_3t+u_4}\right)
 \]
 where $c(\gamma,g)$ is defined in \eqref{Equation:a_0}.
 \end{definition}
 
 We separate the homographies over $\mathbb P^1 (\K)=\K\cup \{\infty\}$ into two sorts. Let $\mathfrak h= \tilde \gamma$ be an homography of $\mathbb P^1(\K)$ associated with an element $\gamma$ as in \eqref{a1a2a3a4}.
 \\ 
 \noindent
 $\bullet$ If the coefficient $u_3$ of $\gamma$ is $0$, then $u_4\not=0$ and we set
\[
q=\frac {u_1}{u_4}, \quad r=\frac {u_2}{u_4}\cdotp
\]
We say that $\mathfrak h$ is an {\it affine homography}. We write it as $\mathfrak h= \mathfrak h_{q,r}$, and for
$t \in \mathbb P^1(\K)$ we have
\[
\mathfrak h (t) = \mathfrak h_{q,r} (t) = qt +r,
\]
with $(q,r) \in \K^\times \times \K$. 
 \\ 
$\bullet$ If $u_3\neq 0$, we set
\[
q=\frac {u_1}{u_3}, \quad r=\frac {u_2u_3-u_1u_4}{u_3^2}, \quad s=-\frac {u_4}{u_3}\cdotp
\]
We say that $\mathfrak h$ is a {\em non affine homography}. We write it as $\mathfrak h= \mathfrak h_{q,r,s}$, and for
$t \in \mathbb P^1(\K)$ we
have
\[
\mathfrak h (t) = \mathfrak h_{q,r,s} (t) = q+\frac r {t-s},
\]
where $(q,r,s) \in \K \times \K^\times \times \K$.

\noindent 
The formulas for the inverses are: 
 \begin{equation}\label{inverse}
 {\mathfrak h}_{q,r}^{-1}= {\mathfrak h}_{q^{-1},-rq^{-1}}
 \quad\text{and}
 \quad
 {\mathfrak h}_{q,r,s} ^{-1}= {\mathfrak h}_{s,r,q}.
 \end{equation}

 From Definition \ref{Definition:h(f)=g} we deduce:
 
 \begin{lemma}\label{998} Let $d \geqslant 2$, let $f(t)$ and $g(z)$ two monic polynomials of $\mathcal P_d(\K)$ written as in \eqref{670*} and \eqref{670**}. Let $\mathfrak h$ be an homography such that $g=\mathfrak h (f)$.
 
\begin{enumerate}
\item If $\mathfrak h$ is an affine homography written as $\mathfrak h = \mathfrak h_{q,r}$, then we have \[
z=\mathfrak h_{q,r} (t) =qt+r, \quad t= \mathfrak h_{q,r}^{-1} (z) = \frac 1q (z-r),
\]
\[
f(t) = \frac 1{q^d} g(qt+r), \quad g(z) = q^d f \left( \frac {z-r}q \right).
\]
\item If $\mathfrak h$ is a non--affine homography written as $\mathfrak h = \mathfrak h_{q,r,s}$, then we have 
\[
z=\mathfrak h_{q,r,s} (t) =q+ \frac r{t-s}, \quad t= \mathfrak h_{q,r,s}^{-1} (z) = s+\frac r{z-q},
\]
\begin{equation}\label{554}
f(t) = \frac{(t-s)^d}{g(q)} g\left( q+\frac r{t-s}\right), \quad g(z) = \frac{(z-q)^d}{f(s)}f \left( s +\frac r{z-q} \right), 
\end{equation}
and
\begin{equation}\label{555}
f(s)g(q) =r^d. 
\end{equation}
\end{enumerate}
\end{lemma}

When $\K=\C$, given a monic polynomial $f$ in $\mathcal P_d(\C)$, we can write 
 \[
f(z)=\prod_{\rho \in \mathcal Z(f)} (z-\rho).
 \] 
From Definition \ref{Definition:h(f)=g} we deduce
\begin{equation}\label{defmathfrakh}
\left(\mathfrak h (f) \right) (z) :=\prod_{\rho \in \mathcal Z (f)} \left( z-\mathfrak h (\rho)\right).
\end{equation}
 Since the polynomials of $\mathcal P_d (\C)$ have exactly $d$ roots which are all distinct, we have, for any homography $\mathfrak h$ and for any $f$ and $g$ in $\mathcal P_d (\C)$ the property
 \[
 \mathfrak h (\mathcal Z (f)) =\mathcal Z (g) \iff g =\mathfrak h (f).
\]

 \begin{lemma} \label{Lemma0.2} Let $d\geqslant 3$ and $F$ and $G$ elements of $\Bin (d, \K)$, written as in \eqref{F=G=}.
 We suppose that $a_0a_db_0b_d \neq 0$. Let $f(t)= F(t,1)/a_0$ and $g(z) = G(z,1)/b_0 $ be the two polynomials in $\mathcal P_d (\K)$ associated with $F$ and $G$. Let $\gamma \in {\rm GL}(2, \K)$ and let $\tilde \gamma= \mathfrak h$ be the homography associated with $\gamma$. Then the two following conditions are equivalent:
 \\
 (i) There exists $\nu\in\K^\times$ such that $G\circ \gamma=\nu F$.
 \\
 (ii) We have the equality
\[
 \mathfrak h (f) =g.
\] 
 \end{lemma}
 
 \begin{proof} 
 $(i)$ implies $(ii)$ has been proved in \eqref{Equation:a_0}.
\\
 Conversely, assume $ \mathfrak h (f) =g$. From Definition \ref{Definition:h(f)=g} we deduce 
 \[
 \begin{aligned}
 F(t,1)=a_0f(t)&=  \frac{a_0}{ c(\gamma,g)}(u_3t+u_4)^dg\left(\frac{u_1t+u_2}{u_3t+u_4}\right)
 \\
 &=\frac {a_0}{b_0 c(\gamma,g)}(u_3t+u_4)^dG\left(\frac{u_1t+u_2}{u_3t+u_4},1\right)
 \\
 &=\frac {a_0}{b_0 c(\gamma,g)}G(u_1t+u_2, u_3t+u_4),
 \end{aligned}
\]
 hence the result with $\nu=\frac{b_0 }{a_0}c(\gamma,g)$.
 \end{proof}

For $\nu\in \K^\times$ and $F\in\Bin(d,\K)$, the polynomial $f$ associated to  $F$  is the same as the polynomial associated to $\nu F$. When $\K=\C$, the two forms $F$ and $\nu F$ are $\C$--isomorphic;  in general, $F$ and $\nu F$ are $\K$--isomorphic when $\nu$ is a $d$--th power of an element in $\K$. 
 Consider for instance the two forms 	
\[
F(X,Y)=X^4+4XY^3-Y^4
\text{ and } G(X,Y)=4F(X,Y).
\]
There is no $\gamma \in  {\rm GL}(2, \Q)$ such that $\tilde \gamma ={\rm Id}$ and $F\circ \gamma =G$. Nevertheless the forms $F$ and $G$ are $\Q$--isomorphic, since we have the equality
  $$F(X+Y, X-Y) = G(X,Y).$$
 
 
\subsection{Preparation of the proof of Theorem \ref{486}} 

By Lemma \ref{278ter} and Proposition \ref{Prop3.9}, we see that a first step in the proof of Theorem \ref{486} will be the following proposition
 
\begin{proposition} \label{central724} Let $d\geqslant 3$. Suppose that there exist 
polynomials $f$ and $g$ in $ {\mathcal P}_d(\C)$ such that
\begin{equation}\label{727}
\Lambda^+ (f) +\Lambda^+ (g) \geqslant d+ 3
\end{equation}
 and a
 homography $\mathfrak h$ such that $\mathfrak h (f) = g$. Then either 
\begin{enumerate}
\item $\mathfrak h$ is a homothety $\mathfrak h_{q,0}$ with $q\in \C^\times$,
or
\item $\mathfrak h$ is a non affine homography and it is of the form $\mathfrak h_{0, r, 0}$ with $r\in \C^\times$. In that situation $f$ and $g$
are of the form $t^d + \alpha_d $ and $z^d+ \beta_d$, with $\alpha_d \beta_d = r^d$.
\item 
The hypothesis \eqref{727} is optimal to obtain the conclusions of the case 2: there exist elements $f$, $g$ of $ {\mathcal P}_d(\C)$ and a non affine homography of the form $\mathfrak h=\mathfrak h_{q,r,s}$, with $(q,s)\not=(0,0)$ with $\mathfrak h (f) = g$ and $\Lambda^+ (f) +\Lambda^+ (g) =d+2$.
\end{enumerate}
\end{proposition} 

\begin{remark} \label{remark:d=3}
In the particular case where $d=3$, the inequality \eqref{727} 
implies that $f$ and $g$ are necessarily of the form $f(t)= t^3+\alpha_3$ and $g(z) = z^3+\beta_3$ with $\alpha_3\beta_3\not=0$. An example of a pair $(f(t), g(z))$ corresponding to the item 3 is the pair
$(t^3+1, z^3+3z)$, since we have $\Lambda^+ (f) + \Lambda^+ (g) =5$ and 
\[
f(t) = \frac{(t-1)^3}{g(1)}g\left( 1+\frac 2{t-1} \right),
\]
(see \eqref{1336} below for a more general construction). \end{remark}

\subsection{Proof of Proposition \ref{central724}.1, the case of affine homographies.}\label{firstcase}

We are concerned with the following question: let $f$ and $g$ two polynomials in $\mathcal P_d (\C)$ satisfying \eqref{727} and written as \eqref{670*} and \eqref{670**}. We want information about the pairs $(q,r)\in \C^\times \times \C$, such that 
\[
\mathfrak h_{q,r} (f) = g.
\]
The hypothesis \eqref{727} and the inequalities \eqref{1leqleqd} imply the inequalities $\Lambda^+ (f) \geqslant 2$ and $ \Lambda^+ (g) \geqslant 2$. This means $\alpha_1=\beta_1 =0$. So the sum of the roots of $
f$ and the sum of the roots of $g$ are equal to $0$. We write 
\[
0= \sum_{\rho \in \mathcal Z (f)} \rho = \sum_{\rho' \in \mathcal Z (g) } \rho' = \sum_{\rho \in \mathcal Z (f)} (q\rho +r) =dr,
\]
so we necessarily have $r=0$. The proof of the item 1 is complete. 
 
\subsection{Proof of Proposition \ref{central724}.2, the case of non affine homographies.} 

The question now is: let $f$ and $g$ be two polynomials in $\mathcal P_d (\C)$ satisfying \eqref{727}. We want information about the triples $(q,r,s)\in \C \times \C^\times\times \C$, such that 
\[
\mathfrak h_{q,r,s} (f) = g.
\]
This question is deeper than the question relative the affine homographies, since the formulas of transformations of the symmetric functions of the roots are more involved.

The case where $f(s)= 0$ corresponds to $g$ having a root sent to infinity. We avoid this fact by considering $g$ to be monic with degree $d$. Similar consideration applies to the condition $g(q)= 0$. Recall that in the definitions of $\mathcal U_d ^{(1)}(\K)$, $\mathcal U_d ^{(2)}(\Z)$, $\mathcal V_d^{(1)}( \K)$ and $\mathcal V_d^{(2)}( \Z)$ (see \S\ref{SS:HomographyFreeSets}), we impose $a_0\neq 0$, so the associated polynomials (see \eqref{669}) are monic with degree $d$. 
 
\subsubsection{From the $f^{(i)}(s) $ to the $\beta_j$}

We now state

 \begin{lemma}\label{hugesystem} Let $d \geqslant 1$ and let $f$ and $g$ be monic polynomials in $\mathcal P_d (\C)$, written as in \eqref{670*} and \eqref{670**}. 
Suppose that there exists $(q,r,s)\in \C \times \C^\times \times \C$ such that the non affine homography $\mathfrak h_{q,r,s}$ satisfies $ g = \mathfrak h_{q,r,s} (f), 
 $ and $f(s) \neq 0$.
 
 Then we have the equalities
 \begin{equation*}
\beta_j f(s)=(-1)^{j} \sum_{i=0}^j (-1)^i q^{j-i} r^i \binom{d-i}{j-i} \frac{f^{(i)}(s)}{i\, !} \ (0\leqslant j\leqslant d)
\end{equation*}
\end{lemma}
\begin{proof} By the last part of formula \eqref{554} of Lemma \ref{998} and by Taylor expansion, one has the equalities
 \begin{align*}
 g(z) & = \frac{(z-q)^d}{f(s)}\sum_{k=0}^d \frac{f^{(k)}(s)} {k\, !} \left(\frac {r}{z-q}\right)^k
 \\ 
 &= \frac{1}{f(s)}\sum_{k=0}^d r^k \frac{f^{(k)}(s)} {k\, !} (z-q)^{d-k}.
 \end{align*}
 Lemma \ref{hugesystem} follows by identification via the binomial formula to expand $(z-q)^{d-k}$.
\end{proof}

\subsubsection{From the $\beta_i$ to the $f^{(j)} (s)$} 
 
\begin{lemma} \label{872} Let $d \geqslant 1$ and let $f$ and $g$ be polynomials in $\mathcal P_d (\C)$, written as in \eqref{670*} and \eqref{670**}.
Suppose that there exists $(q,r,s)\in \C \times \C^\times \times \C$ such that the non affine homography $\mathfrak h_{q,r,s}$ satisfies the equality $ g = \mathfrak h_{q,r,s} (f), 
 $ and $f(s) \neq 0$. 

\noindent Then we have the equalities 
 \begin{equation}\label{expressionofderivatives}\left( \frac rq\right)^j \cdot \frac{f^{(j)}(s)}{j\, !\, f(s)}
=\sum_{i=0}^j \binom{d-i}{j-i} \frac{\beta_{i}}{q^i},
\end{equation}
for $0\leqslant j \leqslant d$.

In the case where $q=0$, this formula has to be interpreted as
\begin{equation}\label{whenq=0}
 r^j \cdot \frac{f^{(j)}(s)}{j\, !\, f(s)}
= \beta_j.
\end{equation}
\end{lemma}

\begin{proof} By the first part of formula \eqref{554} and by \eqref{555}
we have the equalities 
\[
\begin{aligned}
\frac {f(t)}{f(s)}&=\left(\frac{t-s} r\right)^d g\left(q+\frac r {t-s}\right)
\\
&=
\left(\frac{t-s} r\right)^d 
\sum_{i=0}^d \beta_i \left(q+\frac r {t-s}\right)^{d-i}
\\
&=\frac 1 {r^d}
\sum_{i=0}^d \beta_i(q(t-s)+r)^{d-i}(t-s)^i
\\
&= 
\sum_{i=0}^d \beta_i
\sum_{\ell=0}^{d-i} \binom{d-i}{\ell} q^\ell \frac{(t-s)^{\ell+i}}{r^{\ell+i}}\cdotp
\end{aligned}
\]
Hence
\begin{equation}\label{equation:f(t)/f(s)}
\frac {f(t)}{f(s)}=\sum_{j=0}^d \frac{(t-s)^j}{r^j}\sum_{i=0}^j \beta_i\binom{d-i}{j-i} 
 q^{j-i}.
\end{equation}
By Taylor's expansion, we finally get 
\[
\frac{r^jf^{(j)}(s)}{j!f(s)}=\sum_{i=0}^j \beta_i\binom{d-i}{j-i} 
q^{j-i}.
\]
\end{proof}
\begin{remark} 
If we write 
\begin{equation*}
y_i:= \left( \frac rq\right)^i\cdot \frac{f^{(i)}(s)}{i\, !\, f(s)},
\end{equation*}
and
\begin{equation*}
B_j := \frac{ \beta_{j}}{q^j}, \ (0\leqslant j \leqslant d),
\end{equation*}
then Lemmas \ref{hugesystem} and \ref{872} can be written respectively 
\begin{equation}\label{B=Ay}
 \begin{pmatrix} B_0\\ B_1\\
 \vdots \\ 
 B_d
 \end{pmatrix} =
 A
 \begin{pmatrix}y_0\\ y_1\\ \vdots \\ y_d
 \end{pmatrix}
 \quad\text{ and } \quad
 \begin{pmatrix}y_0\\ y_1\\ \vdots \\ y_d
 \end{pmatrix}=
 A^{-1}
 \begin{pmatrix} B_0\\ B_1\\
 \vdots \\ 
 B_d
 \end{pmatrix} 
\end{equation}
 where 
\begin{equation}\label{defAA-1}
 A=
 \Bigl(
 (-1)^{i+j} \binom{d-i}{j-i}\Bigr)_{0\leqslant i,j\leqslant d}
 \quad\text{ and }\quad
 A^{-1}=
 \Bigl(
 \binom{d-i}{j-i}\Bigr)_{0\leqslant i,j\leqslant d}
\end{equation}
and where we extend the definition of the binomial coefficients by setting $\binom n k=0$ for $k<0$.
 
The fact that these two matrices are inverse of each other is the simplest example of {\em inverse relations} (see for instance \cite[p.43--45]{Ri}, where one sign $(-1)^k$ in formula (1) p.~43 should be removed): 
 the inverse of the $(d+1)\times(d+1)$ matrix 
\[
A=\begin{pmatrix}
1&0&0&0&\cdots&0
\\
\displaystyle -\binom d 1&1&0&0&\cdots&0
\\
\displaystyle \binom d 2&\displaystyle -\binom {d-1} 1 &1&0&\cdots&0
\\
\displaystyle -\binom d 3&\displaystyle \binom {d-1}2&\displaystyle - \binom {d-2}1&1&\cdots&0
\\
\vdots&\vdots&\vdots&\vdots&\ddots&\vdots
\\
(-1)^d&(-1)^{d-1}&(-1)^{d-2}&(-1)^{d-3}&\cdots&1
\end{pmatrix}
\] 
is the lower triangular matrix $A^{-1}$ where, in the expression of $A$, we replace each entry by its absolute value.
\end{remark}

The matrices $A^{-1}$ and $A$ enjoy an obvious interpretation through the automorphisms $P(X)\mapsto P(X+1)$ and $P(X)\mapsto P(X-1)$ of the vector space of polynomials with degree $\leqslant d$, equipped with the basis $\{X^d,X^{d-1},\dots,X,1\}$. 
 
\begin{lemma}\label{356bis} We adopt the notations and hypotheses of Lemma \ref{872}. Let $e'$ be an integer such that $0 \leqslant e'\leqslant d-1$. Then we have the equalities
\[
\beta_1= \beta_2 = \cdots = \beta_{e' }=0,
\]
if and only if, we have the equalities
\begin{equation}\label{361bis}
\frac{f^{(k)} (s)}{f(s)} = \frac{d\, !}{(d-k)\, !}\cdot \left( \frac qr\right)^k, \end{equation}
for $1\leqslant k\leqslant e'$.
\end{lemma}
\begin{proof}
This follows from the equalities \eqref{B=Ay}, from the triangular structure of the matrix $A$ and $A^{-1}$
(defined in \eqref{defAA-1}) and from the equality $\beta_0=1$.
\end{proof}

Let us return to the proof of Proposition \ref{central724}.2. 
In the next application of Lemma \ref{356bis}, we exploit the fact that if the integer $k$ is sufficiently large (in terms of $\Lambda^+ (f)$) the $k$--derivative of the polynomial $f$ is a monomial. 

\begin{lemma} \label{896}
Let $d\geqslant 3$ 
and let $f$ and $g$ be polynomials of $\mathcal P_d (\C)$, written as in \eqref{670*} and \eqref{670**}. Let $\Lambda^+(f)$ and $\Lambda^+(g)$ be the two integers defined by \eqref{deflambdas}.
 Suppose that they satisfy
 \begin{equation}\label{e+e>d+1} \Lambda^+ (f) + \Lambda^+ (g) \geqslant d+3,
 \end{equation}
and 
suppose there exists a non affine homography $\mathfrak h_{q,r,s}$ ($(q,r,s) \in \C \times \C^\times \times \C$) such that $\mathfrak h (f) = g$.
 Then we have 
 \[
 q=s=0
 \]
 and 
 \[
 \Lambda^+(f)=\Lambda^+ (g) = d.
 \]
The polynomials $f$ and $g$ have the shapes $f(t) = t^d +\alpha_d$, $g(z) = z^d + \beta_d$ with the relation $\alpha_d \beta_d = r^d$.
\end{lemma}
\begin{proof} The assumptions imply $f(s)\not=0$ (see \eqref{defmathfrakh}). To shorten notations, write $\lambda := \Lambda^+ (f)$ and $\lambda':=\Lambda^+ (g)$. We know from \eqref{1leqleqd} that $1\leqslant \lambda,\, \lambda'\leqslant d$ and from \eqref{670*} that 
\[
f^{(k)} (s) = \frac {d\, !}{(d-k)\, !}\cdot s^{d-k},
\]
for all $d-\lambda +1\leqslant k \leqslant d$. Combining with Lemma \ref{356bis} (with $e'= \lambda'-1$), we deduce that, for all $d-\lambda +1 \leqslant k \leqslant \lambda'-1$, 
one has the equality
\begin{equation}\label{crucialbis} s^{d}
= \left( \frac {qs}r\right)^k\cdot f(s).
\end{equation}
Suppose that the interval $[ d-\lambda +1, \lambda'-1]$ contains two positive consecutive integers $k$ and $k+1$ (this assumption is equivalent to the
 inequality \eqref{e+e>d+1}). We apply \eqref{crucialbis} for these values $k$ and $k+1$ and notice that the left--hand side is constant.

\vskip .3cm
\noindent $\bullet$ {\bf Case $s\neq 0$.} So we have $q\neq 0$ by \eqref{crucialbis}. After division, we have the equality
\[
\frac {qs}r =1.
\]
This equality simplifies \eqref{361bis} into
 \[
 s^k \frac{(d-k)\, !}{d\,!}\cdot \frac{f^{(k)} (s)}{f(s)} = 1, 
 \]
 for $k\leqslant \lambda'-1$. We apply this formula with the choices $k= d-\lambda$ and $k=d-\lambda +1$ and this is legal, since we have 
 $d-\lambda<d-\lambda +1\leqslant \lambda'-1$ as a consequence of the assumption \eqref{e+e>d+1}. Thus we obtain the equalities
 \[
 1=
 \begin{cases}
 \displaystyle{s^{d-\lambda} \cdot \frac{\lambda\,!}{d\, !}\cdot \frac{f^{(d-\lambda)}(s)}{f(s)} = \frac{s^d +\alpha_{\lambda} s^{d-\lambda}\, \lambda \, ! (d-\lambda)\, !/(d\, !)}{f(s)},}\\
 \\
 \displaystyle{s^{d-\lambda +1 }\cdot \frac{(\lambda -1)\,!}{d\, !}\cdot \frac{f^{(d-\lambda +1)}(s)}{f(s)} = \frac{s^d}{f(s)}\cdotp}
 \end{cases}
 \]
 Equating these two expressions and 
 recalling that $\alpha_{\lambda}\neq 0$, we arrive at a contradiction.
 
 \vskip .3cm
 
\noindent $\bullet$ {\bf Case $q\neq 0$.} 
The hypothesis
 \eqref{e+e>d+1} concerning $f$ and $g$ is symmetric. So, by exchanging the r\^oles, we may study the existence of a non affine homography $\mathfrak h_{s,r,q}$ transforming $g$ into $f$ (see \eqref{inverse}). By the above alinea, such a
 homography $\mathfrak h_{s,r,q}$ with $q \neq 0$ does not exist. So we are led to study the remaining case $q=s=0$.

 \vskip .3cm
\noindent $\bullet $ {\bf Case $q=s=0$.} In that case we are asking if, for some $r\neq 0$, the homography 
\[
\mathfrak h_{0,r,0} (t) = \frac rt,
\]
can satisfy $\mathfrak h _{0,r,0} (f) = g$, or in an equivalent form $\mathfrak h_{0, r, 0} (g) =f$. By definition of $\lambda$ we have $\alpha_{\lambda} \neq 0$. 

\noindent $\bullet$ Suppose that $\lambda \leqslant d-1$. Then 
the polynomial $g =\mathfrak h _{0,r,0} (f) $ has its coefficient
$\beta_{d-\lambda}\neq 0$. Since $d-\lambda\geqslant 1$, we deduce the inequality $\lambda '\leqslant d-\lambda$. Such an inequality is incompatible with the hypothesis $\lambda +\lambda'\geqslant d+3$.

\noindent $\bullet$ So we are left with the case $\lambda=d$. Equivalently we have $f(t) = t^d +\alpha_d$, with $\alpha_d \neq 0$. We check that 
$[\mathfrak h_{0,r,0} (f) ](z)= z^d + r^{d}/\alpha_d $ by Lemma \ref{998} \eqref{554}. This is the last statement of Lemma \ref{896}.
\end{proof}
The proof of Proposition \ref{central724}.2 is complete now. 
 
\subsection{Proof of Proposition \ref{central724}.3}\label{discussion} 

\subsubsection{Heuristic considerations.}\label{Heurcons} 
 
Before entering the proof of Proposition \ref{central724}.3 itself, we consider a related general question involving a system of linear non homogeneous equations.

Let $(q,r,s)\in\C\times \C^\times \times \C$, $d\geqslant 3$, $a,b$ with $1\leqslant a, b \leqslant d-1$, $\mathcal I$ and $\mathcal J$ two subsets of $\{1,\dots,d\}$ with respectively $a$ and $b$ elements. Let us consider the following problem. 

\begin{quote}
{\em Does there exist monic polynomials $f,g$ of degree $d$,
\[
f(t)=\alpha_0t^d+\cdots+\alpha_d,\quad g(z)=\beta_0z^d+\cdots+\beta_d
\quad \text{ with }\quad \alpha_0= \beta_0=1
\]
 and discriminants different from zero, which satisfy
 \begin{equation}\label{Equations:IandJ}
 \text{$\alpha_i=0$ for $i\in \mathcal I$ and $\beta_j=0$ for $j\in \mathcal J$ }
 \end{equation}
 and $\mathfrak h_{q,r,s} (f) = g$? }
\end{quote}

Assuming $\mathfrak h_{q,r,s} (f) = g$, there exists $\kappa\in\C^\times$ such that 
\[
\kappa f(t)=(t-s)^d g \left( q+\frac r {t-s}\right),
\]
(to be compared with \eqref{554})
that is 
\[
\kappa f(t)= \sum_{i=0}^d \beta_i(t-s)^i (qt+r-qs)^{d-i}.
\]
Notice that $\kappa f(s)= r^d$. We write \eqref{equation:f(t)/f(s)} as
\[
\frac {f(t)}{f(s)}= \frac 1 {r^d} (q(t-s) + r )^d+\sum_{j=0}^d \frac{(t-s)^j}{r^j}\sum_{i=1}^j \beta_i\binom{d-i}{j-i} 
 q^{j-i}.
 \]
Hence the coefficient of $t^h$ in $\kappa f(t)$ is 
\begin{equation*}
\kappa\alpha_{d-h}= \binom d h q^h (r-qs)^{d-h}+
\sum_{i=\Lambda^+(g)}^d \beta_i \sum_{j=\max\{i,h\}}^d \binom j h \binom {d-i}{j-i}(-s)^{j-h} r^{d-j}q^{j-i}.
\end{equation*}
For $h=d$, since $\alpha_0=1$, this gives $\kappa =g(q)$. 

 The conditions \eqref{Equations:IandJ} 
 are equivalent to a system of $a+1$ linear non homogeneous equations 
\[
\kappa=g(q), \quad \alpha_i=0 \quad (i\in \mathcal I)
\]
with $d-b+1$ unknowns 
\[
\kappa,\quad \beta_j \quad (1\leqslant j\leqslant d,\; j\not\in \mathcal J).
\]

\noindent
{\bf Heuristic.} 
According to the above discussion, subject to the non vanishing of some determinants, we may expect that there is a Zariski closed set of $(q,r,s)$ such that, 
outside this set, 
\\
$\bullet$ when $a+b<d$, then there are infinitely many solutions $(f,g)$;
\\
$\bullet$ when $a+b=d$, there is a unique solution;
\\
$\bullet$ for $a+b>d$, there is no solution.

For instance, given integers $\lambda$ and $\lambda'$ in the interval $[1,d-1]$, the conditions $\Lambda^+(f)\geqslant \lambda$ and $\Lambda^+(g)\geqslant \lambda'$ are a special case of \eqref{Equations:IandJ} with 
\[
\mathcal I=\{1,\dots,\lambda-1\}, \quad \mathcal J=\{1,\dots,\lambda'-1\}, 
\]
$a=\lambda-1$ and $b=\lambda'-1$. 
When $\lambda +\lambda'=d+2$ we may expect that, outside a Zariski closed set of $(q,r,s)$, there is a unique solution. 
In this case item 2 of Proposition 
 \ref{central724} 
shows that $\lambda=\Lambda^+(f)$ and $\lambda'=\Lambda^+(g)$.

\subsubsection{The proof itself.} 

We are now concerned with the proof of the last item of Proposition \ref{central724}. Actually the proof below gives more information. The first step is

\begin{lemma} \label{Lemma:FirstStep}
Let $d \geqslant 3$ and let $(q,r,s)\in(\C^\times)^3$ such that
\begin{equation}\label{rneqneq}
r\neq qs \text{ and } r\neq (d-1) qs.
\end{equation}
Then there exists a unique pair $(f,g)$ of monic polynomials with complex coefficients and with degree $d$, such that
\begin{equation}\label{cond1323}
\Lambda^+ (f)=3,\, \Lambda^+ (g) =d-1 \text { and } 
\mathfrak h_{q,r,s} (f) = g.
\end{equation}
\end{lemma}

 \begin{proof}
 
Write $f$ and $g$ as in \eqref{670*} and \eqref{670**}. By the second and the third conditions of \eqref{cond1323} we can write
\begin{equation}\label{1324}
g(z)=z^d+\beta_{d-1}z+\beta_d, \qquad 
\kappa f(t)= (qt+r-qs)^d+\beta_{d-1} (qt+r-qs)(t-s)^{d-1}+\beta_d(t-s)^d,
\end{equation}
for some complex number $\kappa\neq 0$ (see Lemma \ref{998}.2). 
Since $f$ and $g$ are monic, we have $\alpha_0=\beta_0=1$, that is
\[
\kappa=q^d+q\beta_{d-1}+\beta_d.
\]
The inequality $\Lambda^+(f)\geqslant 3$ is equivalent to the double equality $\alpha_1=\alpha_2=0$. So we are led to consider 
 the system of three non--homogeneous linear equations in three unknowns $\beta_{d-1}$, $\beta_d$ and
 $\kappa$
\begin{equation}\label{equation:system}
\left\{
\begin{matrix}
\hfill
q\beta_{d-1}&+&
\hfill\beta_d&-&\kappa&=&-q^d\hfill
\\
\hfill
(r-dqs)\beta_{d-1}&-&\hfill
ds\beta_d&&&=&-dq^{d-1}(r-qs)\hfill
\\
\hfill
(dqs-2r)s\beta_{d-1}&+&\hfill
d s^2\beta_d&&&=&-dq^{d-2}(r-qs)^2.\hfill
\end{matrix}
\right.
\end{equation}
The determinant is 
\[
\det\begin{pmatrix}
q&1&-1 
\\
r-dqs&-ds&0 
\\
(dqs-2r)s&d s^2&0
\end{pmatrix}=drs^2.
\]
Since $qrs\not=0$,
there is a unique solution
\[
\beta_{d-1}=d q^{d-2}\left(\frac r s -q\right),
\quad
\beta_d=q^{d-2}\left(\frac r s -q\right)\left(\frac r s -(d-1)q\right),
\qquad
\kappa=q^{d-2}\left(\frac r s \right)^2.
\]
The assumption \eqref{rneqneq} ensures $\beta_{d-1}\not =0$ and $\beta_d \neq 0$. 

Let us check that $\alpha_3\not=0$, which means $\Lambda^+(f)=3$. By considering the coefficient of $t^{d-3}$ on both sides of the last equality of \eqref{1324} and by the above values of $\beta_d$ and $\beta_{d-1}$ we obtain 
\[
\begin{aligned}
\kappa\alpha_3
&= \binom d 3 q^{d-3}(r-qs)^3+\beta_{d-1} \binom {d-1} 2 s^2\left( r-\frac d 3 qs\right)-\beta_d \binom d 3 s^3 
\\
&= \binom d 3 q^{d-3}(r-qs)^3+ \binom {d-1} 2 dq^{d-2}s\left( r-\frac d 3 qs\right) (r -qs) \\
& \qquad \qquad \qquad \qquad \qquad \qquad \qquad \qquad - \binom d 3 s q^{d-2}( r -qs)\bigl( r -(d-1)qs\bigr),
\\
&= \binom d 3 q^{d-3}r^2(r-qs)\not=0,
\end{aligned}
\]
by assumption. 
 \end{proof}

 Here is how to check that the polynomial $g(z)$ as defined in \eqref{1324} with the above values of $\beta_d$ and $\beta_{d-1}$ has a discriminant different from zero, which is a necessary condition for $g(z)$ to belong to $\mathcal P_d (\C)$. Assume there exists $\rho \in \C$ such that $g(\rho) = g' (\rho) =0$. The condition $g'(\rho) =0$ is equivalent to 
 \begin{equation}\label{Equation:496}
\rho^{d-1} = - \frac{\beta_{d-1}}{d}\cdotp
\end{equation}
Combining this value with $g(\rho) =0$ yields
\[
0 = g(\rho) = \rho \left( - \frac {\beta_{d-1}} d\right) + \beta_{d-1} \rho + \beta_d.
\]
Hence, if $\rho$ exists, its value can only be 
\[
\rho = \frac d{1-d}\cdot \frac{\beta_d}{\beta_{d-1}}= \frac{r/s-(d-1)q}{1-d},
\]
thanks to the values already computed for $\beta_d$ and $\beta_{d-1}$. From \eqref{Equation:496} we get the condition for the discriminant of $g$ to be nonzero:
\[
\left( \frac{r/s -(d-1) q}{1-d}\right)^{d-1} \neq -q^{d-2} (r/s-q).
\]
 To complete the proof of the item 3 of Proposition \ref{central724}, we produce an instance of two elements $f$, $g$, of $\mathcal P_d(\Q)$ (i.e. with nonzero discriminant) with $\Lambda^+(f)=3$ and $\Lambda^+(g)=d-1$ and $\mathfrak h_{q,r,s} (f) = g$. Thanks to Remark \ref{remark:d=3} we may assume $d\geqslant 4$. 
Take 
$q=1$, $r=2$, $s=1$, $\beta_{d-1}=d$, $\beta_d=3-d$, $\kappa=4$, in other words we have
\[
\gamma=\begin{pmatrix} 1&1\\1&-1\end{pmatrix}, \quad g(z)=z^d+dz-d+3,
\]
so that $g(q)=4$, 
\begin{equation}\label{1336}
f(t)=\frac{(t-1)^d}{g(q)} g\left(\frac {t+1}{t-1}\right)=\frac 1 4 \Bigl(
(t+1)^d +d(t+1)(t-1)^{d-1}+(3-d)(t-1)^d\Bigr).
\end{equation}
One checks $\alpha_0=1$, $\alpha_1=\alpha_2=0$, $\alpha_3=\displaystyle \binom d 3 \not=0$, 
\[
f(t) = t^d+ \binom{d}{3} t^{d-3} + \alpha_4t^{d-4}+ \cdots+\alpha_d,
\]
$f(1)=2^{d-2}$, $f(s)g(q)=r^d$, $\Lambda^+(f)=3$. 

The derivative of $g$ is $g'(z)=d(z^{d-1}+1)$. Let $\rho$ be one its roots. It satisfies $\rho^d =-\rho$ so we have the equality
\[
g(\rho) = (d-1) \rho -d+3,
\] 
and $g(\rho)$ does not vanish, since $\rho$ has modulus one and $d\geqslant 4$. Since $g(z)$ and $g'(z)$ do not vanish simultaneously,
the polynomial belongs to $\mathcal P_d (\C)$. The proof of the item 3 of Proposition \ref{central724} is complete.

\subsubsection{Comments}

In Lemma \ref{Lemma:FirstStep} we assume $qs\not=0$. 
There is no example of pair $(f,g)$ of monic polynomials of degree $d$ satisfying \eqref{cond1323}
with $qs=0$, $\Lambda^+(f)=3$ and $\Lambda^+(g)=d-1$. 
Indeed for $q=0$ the system \eqref{equation:system} yields
\[
\left\{
\begin{matrix}
r\beta_{d-1}-ds\beta_d&=0
\\
-2rs\beta_{d-1}+ds^2\beta_d&=0
\end{matrix}
\right.
\]
which has no solution satisfying $r\not=0$ and $\beta_{d-1}\not=0$, while for $s=0$ and $q\not=0$ this system \eqref{equation:system} yields $dq^{d-2}r^2=0$, which is not allowed.

We now present some examples of pairs $(f,g)$ of monic polynomials of degree $d$ satisfying $\mathfrak h_{q,r,s} (f) = g$ and $\Lambda^+(f) +\Lambda^+ (g)<d+2$.

\begin{enumerate}
\item Here is an example with $q=0$, $s\not=0$, $\Lambda^+ (f) =2$, $\Lambda^+ (g) =d-1$:
\[
g(z)=z^d+\frac {ds}r z+1,\qquad f(t)=(t-s)^d g\left(\frac r {t-s}\right)=(t-s)^d+ds(t-s)^{d-1}+r^d.
\]

\item When $\lambda+\lambda'=d$, explicit solutions $(f,g)$ with $\Lambda^+(f)=\lambda$ and $\Lambda^+(g)=\lambda'$ are given with $q=s=0$ by trinomial forms
\[
f(t)=t^d+ \alpha_\lambda t^{d-\lambda}+\alpha_d,
\qquad
g(z)=z^d+\beta_{d-\lambda}z^\lambda +\beta_d
\]
with
\[
\alpha_\lambda \beta_d = \beta_{d-\lambda}r^\lambda, \ \ \alpha_d \beta_d = r^d.
\]
\item When $\lambda+\lambda'<d$, explicit solutions are given with $q=s=0$ by quadrinomial forms
\[
f(t)=t^d+ \alpha_\lambda t^{d-\lambda}+\alpha_{d-\lambda'} t^{\lambda'}+\alpha_d,
\qquad
g(z)=z^d+\beta_{\lambda'}z^{d-\lambda'} +\beta_{d-\lambda}z^\lambda +\beta_d
\]
with
\[
\alpha_d\beta_{\lambda'}=\alpha_{d-\lambda} r^{\lambda'}, \quad
\alpha_d\beta_{d-\lambda}=\alpha_{\lambda} r^{d-\lambda}, \quad \alpha_d\beta_d=r^d.
\]
 
 \end{enumerate} 
 
 \section{Proofs of Theorem \ref{486} and Corollary \ref{Corollary2star}}\label{Proofs486et2star}
 
\subsection{Proof of Theorem \ref{486}}\label{1033} 

Let $F$ and $G$ be two forms of the reduced set $\mathcal E$ of $ \Bin (d, \K)$. 
We suppose that they are written as in \eqref{F=G=} with $a_0a_db_0b_d \neq 0$. 
By assumption we have $\min\{\Lambda^+ (F), \, \Lambda^+ (G)\}\geqslant (d+3)/2$.
 Let $f$ be the monic polynomial associated to $F$ defined by \eqref{669} and written as in \eqref{670*} with $\alpha_i=a_i/a_0$ for $0\leqslant i\leqslant d$ and similarly let
$g$ be the monic polynomial associated to $G$ written as \eqref{670**} with $\beta_i=b_i/b_0$ for $0\leqslant i\leqslant d$. 
They satisfy \eqref{727}. Our aim is to prove that if there exists a matrix $\gamma$ (written as in \eqref{abcd}) such that $F\circ \gamma
 =G$, then $F=G$. Assume such a $\gamma$ exists and let $\tilde \gamma =\mathfrak h$, the homography attached to $\gamma$. 
Replacing $\gamma$ with $\gamma^{-1}$ in Lemma \ref{Lemma0.2} we deduce $\mathfrak h (g) =f$. 
 
 \vskip .2cm
 \noindent Suppose that $u_3\neq 0$, then $\mathfrak h$ is a non affine homography exchanging $\mathcal Z(f)$ and $\mathcal Z (g)$.
 By Proposition \ref{central724}.2, we deduce that $\mathfrak h$ is of the form $\mathfrak h_{0,r,0}$ 
 and that both $F$ and $G$ are binomials. Thus $\gamma$ satisfies $u_1=u_4=0$. The hypothesis $F\circ \gamma =G$ leads to the equality $F(u_2Y, u_3 X) = G(X,Y)$. This contradicts the item 3 of Definition \ref{Def:reduced}.
 
 \vskip .2cm
 \noindent Hence $u_3=0$ and $\mathfrak h$ is an affine homography exchanging $\mathcal Z(f)$ and $\mathcal Z(g)$. 
 By Proposition \ref{central724}.1, the homography $\mathfrak h$ has to be a homothety. The matrix $\gamma$ satisfies $u_2=u_3=0$. The hypothesis $F\circ \gamma =G$ leads to the equality $F(u_1X, u_4Y) =G(X,Y)$. Item 2 of Definition \ref{Def:reduced} implies $F=G$. 
 
 \vskip .2cm
\noindent 
This completes the proof of Theorem \ref{486}. 

 \subsection{Proof of Corollary \ref{Corollary2star}} \label{proof2.2} 
 
 The proof will use two auxiliary results.
 
 \begin{lemma}\label{Lemma:Wd1}
 Let $d\geqslant 3$ and $\mathcal W_d^{(1)}(\K)$ be a subset of $\Bin (d,\Z)$ such that for any $F\in \mathcal W_d^{(1)}(\K)$ we have
 $a_0\not=0$, $a_{d-1} =a_d =1$. 
 Then the set $\mathcal W_d^{(1)}(\K)$ is $\K$--dilation free. 
 \end{lemma}
 
 \begin{proof}
 Let $F$ and $G$ be two elements in $\mathcal W_d^{(1)} (\K)$ written as in \eqref{F=G=} 
and let $\mathfrak h_{q,0}$ be a homothety which exchanges the zeroes of the polynomials $f$ and $g$ associated to $F$ and $G$ written as 
\eqref{670*} and \eqref{670**}. By Lemma \ref{iso<->auto}, 
all the 
$\gamma\in {\GL}(2, \K)$ such that $\tilde \gamma = \mathfrak h_{q,0}$ and $F\circ \gamma = G$
have the shape 
\begin{equation}\label{t00u}
\gamma = \begin{pmatrix} u& 0\\ 0 & v
\end{pmatrix},
\end{equation}
where $u$ and $v$ are complex numbers different from zero. Returning to the explicit expression of $F$ and $G=F\circ \gamma$
we have 
\[
G(X,Y)=F(uX,vY)=a_0u^dX^d+a_1u^{d-1}vX^{d-1}Y+\cdots+a_{d-1}uv^{d-1}XY^{d-1}+a_dv^dY^d,
\]
hence the equalities $1=b_d = v^d a_d =v^d$ and $1= b_{d-1} =uv^{d-1} a_{d-1} =uv^{d-1}$. They imply 
that $u=v$ and $u^d=1$. So $G(X,Y) = F(uX, uY) = u^d F(X,Y) = F(X,Y)$. Hence $\mathcal W_d^{(1)} (\K)$ is $\K$--dilation--free.
 \end{proof}

 \begin{lemma}\label{Lemma:rigid}
 Let $d\geqslant 3$ and $\mathcal W_d^{(2)}(\Z)$ be a subset of $\Bin (d,\Z)$ such that for any $F\in \mathcal W_d^{(2)}(\Z)$ written as in \eqref{defF} we have
 \begin{enumerate}
 \item $a_0>0$, $a_d\not=0$, 
$a_0$ and $a_d$ are $d$--free,
 \item if there is an odd index $k$ such that $a_k\not=0$, then for the smallest such $k$ we have $a_k>0$.
 \end{enumerate}
 Then
 \\
 (a) The set $\mathcal W_d^{(2)}(\Z)$ is $\Q$--dilation free. 
 \\
 (b)
 Let $F\in \mathcal W_d^{(2)}(\Z)$ and let
\[
 \gamma = \begin{pmatrix} u& 0\\ 0 & v
\end{pmatrix} \in {\GL}(2, \Q)
\]
be such that $F\circ\gamma=F$. Then 
 \[
 \gamma = 
 \begin{cases} \pm {\rm Id} & \text{ if $F$ is not a binary form with squared arguments,}
 \\ 
\displaystyle \pm {\rm Id}\; \text { or } \; \pm \begin{pmatrix}1 & 0\\ 0 & -1 \end{pmatrix} &\text{ if $F$ is a binary form with squared arguments.}
 \end{cases}
 \]
 \end{lemma}
 
 \begin{proof}
 (a)
Let $F$ and $G$ be two elements in $\mathcal W_d^{(2)} (\Z)$ written as in \eqref{F=G=} and $u,v$ two nonzero rational numbers such that $F(uX,vY)=G(X,Y)$. Since the coefficients $a_0$ and $b_0$ of $X^d$ in $F$ and $G$ respectively are $d$--free integers, and since $a_0>0$ and $b_0>0$, the equality $a_0u^d=b_0$ implies $u^d=1$, and similarly $v^d= \pm 1$. 

\noindent
 -- If $u=v \, (=\pm 1)$ then $F=G$.
 \\
 -- If $u=1$ and $v=-1$, then $F(X,Y)$ and $F(X,-Y)$ belong to $\mathcal W_d^{(2)} (\Z)$. If there is some odd $k$ such that $a_k\neq 0$, then the least such $k$ satisfies $a_k>0$. Since $F(X,-Y)$ also belongs to $\mathcal W_d^{(2)} (\Z)$, we deduce that $(-1)^k a_k$ is also positive. This gives a contradiction. So both $F$ and $G$ have squared arguments. They are equal.
 \\
 -- If $u=-1$ and $v=1$, then $d$ is even necessarily. Now $F(X,Y)$ and $F(-X, Y)$ both belong to $\mathcal W_d^{(2)} (\Z)$. If there is some odd $k$ such that $a_k\neq 0$, then the least such $k$ satisfies $a_k>0$. Since $F(-X,Y)$ also belongs to $\mathcal W_d^{(2)} (\Z)$, we deduce that $(-1)^{d-k} a_k$ is also positive. The end of the proof is as above. 
 \end{proof}
 
 \begin{proof}[Proof of Corollary \ref{Corollary2star} ]
To prove the first item we are going to use Theorem \ref{486}. 
We first check item 3 in the Definition \ref{Def:reduced} of reduced sets. 
Since for $d\geqslant 3$ a binary form $F$ such that $1\leqslant \Lambda^+(F)\leqslant d-1$ is not a binomial form, it follows that no binomial binary form belongs to $\mathcal U_d^{(1)} (\Z)$ nor to $\mathcal U_d^{(2)} (\Z)$. 

Lemma \ref{Lemma:Wd1} shows that the set $\mathcal U_d^{(1)} (\K)$ is $\K$--dilation free. 

From Theorem \ref{486} we deduce that the set $\mathcal U_d^{(1)} (\K)$ is $\K$--homography--free.

Consider the second item of Corollary \ref{Corollary2star}. We need to check that the set $\mathcal U_d^{(2)} (\Z)$ is $\Q$--homography--free, and Theorem \ref{486} shows that it suffices to check that it is $\Q$--dilation free. This result follows from the first part (a) of Lemma \ref{Lemma:rigid}.

To complete the proof of the second item of Corollary \ref{Corollary2star} 
it only remains to be checked that the elements of $\mathcal U^{(2)}_d (\Z)$ are $\Q$--rigid binary forms, which means that if $F $ belongs to $\mathcal U^{(2)}_d (\Z)$, then its group of $\Q$--automorphisms satisfies \eqref{299}. 

By Proposition \ref{central724}.1, 
which is also valid for $f=g$, the only affine homographies which permute $\mathcal Z (f)$ are $\Q$--homotheties. 
By the remark made in \S \ref{1033}, 
any $\Q$--automorphism $\gamma$ of $F$ is such that 
$\tilde \gamma$ is a $\Q$--homothety. So $\gamma$ has the shape \eqref{t00u}, but with $u$ and $v$ rational numbers different from zero. 
It only remains to apply part (b) of Lemma \ref{Lemma:rigid}. 

 The proofs of both items of Corollary \ref{Corollary2star} are complete.
 \end{proof}
 
 
\subsection{Not homothetic pairs of polynomials}

The first alinea of Proposition \ref{central724} gives no information about the existence or not of a homothety $\mathfrak h_{q,0}$ exchanging the distinct polynomials $f$ and $g\in \mathcal P_d (\K)$.
Here we give examples (without proofs) of subsets of $\mathcal P_d (\K)$ (defined in \eqref{defPd(K)}) for which such homotheties do not exist. 

\begin{example} 
For $1\leqslant k <d$, let 
\[
\mathcal P_{d,k, k+1}(\C) :=
\{ f \in \mathcal P_d (\C) : \alpha_k = \alpha_{k+1} \neq 0
\}.
\]
Then there is no $q\in \C\smallsetminus \{0, 1\}$ and no pair $(f,g)$ of (distinct or not) elements of $ \mathcal P_{d,k, k+1}(\C)$, such that 
\[
\mathfrak h_{q,0} (f) =g.
\]
\end{example}

\begin{example}\label{nohomothetyQ} 
The second example is of arithmetical nature. Notice that since the discriminant of the elements in $\mathcal P_d (\Z) $ is not $0$, one at least of the two coefficients $a_d$, $a_{d-1}$ is not $0$.
For $2\leqslant k \leqslant d$, consider the set
\[
{\mathcal P}_{d,k,{\rm free}}(\Z):= \left\{ f \in \mathcal P_d (\Z) : a_k\neq 0 \text{ and } a_k \text{ is } k-\text {free}\right\}.
\]
Let $d\geqslant 4$. There is no $q\in \Q\smallsetminus\{0,1,-1\}$ and no distinct $f$ and $g\in {\mathcal P}_{d,k,{\rm free}}(\Z)$ such that 
\begin{equation*}
\mathfrak h_{q,0} (f) =g.
\end{equation*}
\end{example}

\begin{example} 
To eliminate the case of symmetry $\mathfrak h_{-1, 0}$, it is sufficient to consider the following subset of ${\mathcal P}_{d,k,{\rm free}}(\Z)$ defined by
 \[
 {\mathcal P}_{d,k,{\rm free}}^+(\Z):= \left\{ f \in {\mathcal P}_{d,k,{\rm free}}(\Z): \alpha_k >0 \right\} .
 \]

Here is the variant of Example \ref{nohomothetyQ}:

 Let $d\geqslant 4$, $k$ odd satisfying $2\leqslant k\leqslant d$. There is no $q \in \Q \smallsetminus \{0, 1\}$ and no distinct $f$ and $g\in {\mathcal P}_{d,k,{\rm free}}^+(\Z)$ such that
\[
\mathfrak h_{q,0} (f) =g.
\]
\end{example}
 
\section{Proof of Theorem \ref{527} and Corollary \ref{Corollary2}} \label{S:Proof527}
We now write $F$ and $G$ as 
\begin{equation*}
\begin{cases}
F(X,Y) = a_0X^{d} + a_{\lambda} X^{d-\lambda} Y^{\lambda} + a_{\lambda +5} X^{d-\lambda-5} Y^{\lambda +5}+a_{\lambda +6} X^{d-\lambda -6} Y ^{\lambda +6}+ \cdots + a_{d} Y^{d}, \\
\\
G(X,Y) = b_0X^{d} + b_{\lambda'} X^{d-\lambda'} Y^{\lambda'} + b_{\lambda' +5} X^{d-\lambda'-5} Y^{\lambda' +5}+b_{\lambda' +6} X^{d-\lambda' -6} Y ^{\lambda' +6}+ \cdots + b_{d} Y^{d}, 
\end{cases}
\end{equation*}
where $\lambda =\Lambda^+ (F), $ $\lambda' = \Lambda^+ (G)$, $ a_0a_db_0b_da_{\lambda}b_{\lambda'}\neq 0$.
By analogy with the above section, we introduce the two associated polynomials (see \eqref{669}) 
\begin{equation}\label{twopolynomials}
\begin{cases}
f(t) =t^{d}+\alpha_{\lambda}t^{d-\lambda} + \alpha_{\lambda +5} t^{d-\lambda -5} +\alpha_{\lambda +6} t^{d-\lambda -6} +\cdots +\alpha_{d},\ (\alpha_{\lambda}\alpha_d\neq 0)\\
\\
 g(z) =z^{d}+\beta_{\lambda'} z^{d-\lambda'} + \beta_{\lambda' +5} z^{d-\lambda' -5} +\beta_{\lambda'+6} z^{d-\lambda'-6} +\cdots +\beta_{d}, \, (\beta_{\lambda'} \beta_d \neq 0),
 \end{cases}
\end{equation}
with $\alpha_k =a_k/a_0$ and $\beta_k = b_k/b_0$ for $k=1,\dots, d$; we also set $\alpha_0=\beta_0=1$.
So the polynomials $f$ and $g$ have (at least) four coefficients equal to zero just after the second monomial 
(the second monomial of $f$ is $\alpha_{\Lambda^+ (f)} t^{d-\Lambda^+ (f)}$, the first monomial is $t^d$)
and we have
\begin{equation}\label{1144}
\max\{\lambda, \, \lambda'\} \leqslant d-5.
\end{equation}

 The analogue of Proposition \ref{central724} is the following.
 
\begin{proposition} \label{central} Let $d\geqslant 10$. Consider in $\C [t]$ and $\C [z]$ respectively, two polynomials $f$ and $g$ (distinct or not)
as in \eqref{twopolynomials}, satisfying \eqref{1144} and the inequality
\begin{equation}\label{1151}
\lambda + \lambda ' \geqslant d.\end{equation}
Suppose that the discriminants of $f$ and $g$ do not vanish and suppose 
 the existence of a
 homography $\mathfrak h$ such that $\mathfrak h (f) = g$.
 Then either 
\begin{enumerate}
\item $\mathfrak h$ is a homothety $\mathfrak h_{q,0}$ with $q\in \C^\times$,
or
\item $\mathfrak h$ is a non affine homography of the form $\mathfrak h_{0, r, 0}$ with $r\in \C^\times$. In that situation we have 
\begin{equation}\label{1159}
\lambda +\lambda' =d
\end{equation}and $f$ and $g$
are of the form $t^d +\alpha_{\lambda} t^{d-\lambda} + \alpha_d $ and $z^d+ \beta_{\lambda'} z^{d-\lambda'}+ \beta_d$, with $\alpha_d \beta_d =r^d$, 
 and $ \alpha_\lambda r^{\lambda'}= \alpha_d \beta_{\lambda'}$.
\end{enumerate}
\end{proposition}
 
\subsection{Heuristics again}

 Let $(q,r,s)\in\C\times \C^\times \times \C$ and $d\geqslant 3$, $\lambda,\lambda',\mu,\mu'$ positive integers. We are looking for the existence of two polynomials $f,g$ in $\mathcal P_d$ satisfying 
\[
\kappa f(t)=(t-s)^d g \left( q+\frac r {t-s}\right)
\]
for some $\kappa\in\C^\times$ 
and 
 \[
 \begin{aligned}
 &\alpha_1=\cdots=\alpha_{\lambda-1}=\alpha_{\lambda+1}=\cdots=\alpha_{\lambda+\mu-1}=0
 \\ 
 &\beta_1=\cdots=\beta_{\lambda'-1}=\beta_{\lambda'+1}=\cdots=\beta_{\lambda'+\mu'-1}=0,
 \end{aligned}
 \]
 so that
 \[
 f(t)= t^d+\alpha_{\lambda}t^{d-\lambda}+
\sum_{j=\lambda+\mu}^d \alpha_jt^{d-j}
\]
and
\[
g(z)= z^d+ \beta_{\lambda'}z^{d-\lambda'}+
\sum_{j=\lambda'+\mu'}^d \beta_jz^{d-j},
\] 
(compare with \eqref{twopolynomials}). According to the heuristic discussion of \S \ref{Heurcons}
with the two subsets 
\[
\mathcal I=\{1,\dots,\lambda-1,\lambda+1,\dots,\lambda+\mu-1\}, \quad \mathcal J=\{1,\dots,\lambda'-1,\lambda'+1,\dots,\lambda'+\mu'-1\}, 
\]
with $a=\lambda+\mu-2$ and $b=\lambda'+\mu'-2$, we may expect that, given $d\geqslant 3$, $\lambda,\lambda',\mu,\mu'$ satisfying 
 $\lambda +\lambda'+\mu+\mu'=d+4$, outside a Zariski closed set of $(q,r,s)$, there is a unique solution. We work out an example for sufficiently large $d$ with $\lambda=\mu=2$, $1\leqslant \lambda'\leqslant d-1$, $\mu'=d-\lambda'$ below (Example \ref{Example}). 
 
 When $\lambda +\lambda'+\mu+\mu'>d+4$, we may expect that there is no solution. The main result of Proposition 
\ref{central}
is that this conclusion holds under the stronger assumptions $\mu\geqslant 5$, $\mu'\geqslant 5$, $\lambda +\lambda'\geqslant d$. 

\begin{example}\label{Example} 
Take 
$q=1$, $r=1$, $s=2$, $1\leqslant \lambda'\leqslant d-1$, $\mu'=d-\lambda'$, 
\[
\gamma=\begin{pmatrix} 1&-1\\1&-2\end{pmatrix}, \quad g(z)=z^d+\beta_{\lambda'}z^{d-\lambda'}+\beta_d,
\]
\[
\kappa f(t)= (t-1)^d+\beta_{\lambda'} (t-1)^{d-\lambda'}(t-2)^{\lambda'}+\beta_d(t-2)^d.
\] Since $f$ is monic, we have
\[
\kappa=1+\beta_{\lambda'} +\beta_d.
\]
We are searching for $\beta_{\lambda'}$, $\beta_d$ and $\kappa$ such that $\alpha_1=\alpha_3 =0$, so that 
\[
f(t)=t^d+\alpha_2t^{d-2}+\alpha_4t^{d-4}+\cdots+\alpha_d. 
\]
We have 
\[
\begin{aligned}
\kappa\alpha_1&=-d+A\beta_{\lambda'}-2d\beta_d,
\\
\kappa\alpha_3&=-\binom d 3 +B\beta_{\lambda'}-8 \binom d 3 \beta_d
\end{aligned}
\]
with 
\[
\begin{aligned}
A&= -(d+\lambda'),
\\
B&=
-8\binom {\lambda'} 3 - 4\binom {\lambda'} 2(d-\lambda')-2\lambda' \binom {d-\lambda'} 2- \binom {d-\lambda'} 3.
\end{aligned}
\]
The determinant of the system of three equations in thee unknowns $\beta_{\lambda'}$, $\beta_d$ and $\kappa$
\[
\left\{
\begin{matrix}
\hfill
 \beta_{\lambda'}&+&
\hfill\beta_d&-&\kappa&=&-1\hfill
\\
\hfill
A\beta_{\lambda'} &-&\hfill
2d\beta_d&&&=&d \hfill
\\
\hfill
B\beta_{\lambda'}&-&\hfill
\displaystyle
8 \binom d 3 \beta_d&&&=&
\displaystyle \binom d 3,\hfill
\end{matrix}
\right.
\]
is 
\[
\Delta := \begin{vmatrix} 
1& 1& -1\\
A& -2d& 0\\
B& -8 \binom d3 & 0
\end{vmatrix}
=8A\binom d3 -2d B; 
\]
this is a polynomial in $d$ of degree $4$ and which is multiple of $d$. It vanishes for at most three values of $d\not=0$.
In particular for $d$ sufficiently large it
is different from zero, thus this system has unique solution $(\kappa,\beta_{\lambda'},\beta_d)$. We want to study the potential vanishing of the unknowns, as the parameter $d$ tends to infinity, when the other parameter $\lambda'$ is a fixed positive integer. 
Standard computations lead to the equality
\[
B= \frac 16 \left( -d^3-3d^2(\lambda' -1)+d (-3 \lambda'^2+12 \lambda' -2)-\lambda'(\lambda'^2 -9 \lambda +14)
\right),
\]
which is summarized as 
\[
B= - \frac {d^3} 6 -\frac {d^2}2 (\lambda'-1) +O(d).
\]
We use the formula
\[
\binom d3 = \frac{d^3}6 -\frac{d^2}2 +O(d).
\]
$\bullet$ To study the value of $\beta_{\lambda'}$, we consider the determinant
\[
\Delta=\begin{vmatrix} -1& 1 &-1\\
d &-2d& 0\\
\binom d3& -8 \binom d3 &0
\end{vmatrix} = d \binom d3 \begin{vmatrix}
1& 2\\1&8
\end{vmatrix} = 6d \binom d3= d^4 + O(d^3).
\]
So we obtain that $\beta_{\lambda'} \longrightarrow -1$ as $d\longrightarrow +\infty$.
\\
$\bullet$ To study the value of $\beta_d$, we consider the determinant 
\[
\begin{vmatrix} 1& -1 &-1\\
A &d& 0\\
B& \binom d3 &0
\end{vmatrix} = - A\binom d3 +Bd =\frac{ \lambda'}3 d^3 +O (d^2).
\]
Using the above value of $\Delta$, we deduce that, for large $d$, $\beta_d$ tends to zero, without vanishing.
\\
$\bullet$ To study the value of $\kappa$, we already know that $\kappa$ tends to zero as $d$ tends to infinity: this follows from the first equation of the system. To prove that $\kappa$ does not vanish for large values of $d$ we compute the determinant
\[
\begin{vmatrix}1 & 1 &-1 \\
A & -2d & d\\
B & -8 \binom d3 & \binom d3
\end{vmatrix} = \begin{vmatrix}
0 & 0 & -1\\
A+d & -d & d\\
B+ \binom d3 & -7 \binom d3 & \binom d3
\end{vmatrix} =-7(A+d) \binom d3 +d ( B + \binom d3), 
\]
which is finally equal to
\[
\frac {2\lambda'}3 d^3 +O(d^2).
\]
By dividing by the value of $\Delta$, we find that $\kappa$ tends to zero, without vanishing.

Using
\[
\binom d 2-4\binom{\lambda'}2 -2\lambda'(d-\lambda')-\binom {d-\lambda'}2=-\frac {\lambda'}2 (2d+\lambda'-3)
\]
we deduce that when $d$ tends to infinity, we have 
\[
\alpha_2\to -\frac {\lambda'}2 (2d+\lambda'-3)\not=0.
\]
Also from 
\[
\begin{gathered}
\binom d 4-16\binom {\lambda'}4-8\binom {\lambda'}3(d-\lambda')-4\binom b 2 \binom {d-\lambda'}2-2\lambda'\binom {d-\lambda'}3 -\binom {d-\lambda'}4 =
\\
-\frac 1{24} \lambda' (-90 + 4 d^3 + 6 d^2 (-5 + \lambda') + 83 \lambda' - 18 \lambda'^2 + \lambda'^3 + d (82 - 42 \lambda' + 4 \lambda'^2))
\end{gathered}
\]
we conclude that $\alpha_4\not=0$ for sufficiently large $d$.

This completes the claim of Example \ref{Example} that for $\lambda=\mu=2$, $1\leqslant \lambda'\leqslant d-1$ and $\mu=d-\lambda'$,
 there is an example of a pair of polynomial $f,g$ in $\mathcal P_d$ and an homography $\mathfrak h$ satisfying $\mathfrak h(f)=g$. 

\end{example}
\subsection{Proof of Proposition \ref{central}} 

Assume first that $\mathfrak h$ is an affine homography $\mathfrak h_{q,r}$ with $q\in \C^\times$. We use the same argument as in the proof given in \S \ref{firstcase}. The conditions \eqref{1144} and \eqref{1151}
imply that $\lambda \geqslant 2$ and $ \lambda' \geqslant 2$, 
thus the sums of the roots of $f$ and $g$ are equal to zero. Hence $r=0$ and $\mathfrak h$ is a homothety $\mathfrak h_{q,0}$.

\vskip .2cm
\noindent Now consider the case where $\mathfrak h$ is a non affine homography: $\mathfrak h=\mathfrak h_{q,r,s}$ with $r\in \C^\times$. The proof below works by contradiction. We will show that each of the three cases $\mathfrak h_{q,r,s}$ with $q$ and $s\neq 0$, with $q= 0$ and $s\neq 0$ and finally with $q\neq 0$ and $s= 0$ are impossible.

To start with, we consider the general case $(q,r,s)\in\C\times\C^\times\times\C$. 

As a consequence of Lemma \ref{896}, we suppose that $\lambda$ and $\lambda'$ satisfy the inequalities
\begin{equation}\label{lambda+lambda}
d\leqslant \lambda + \lambda ' \leqslant d+2.
\end{equation}
 Thanks to Lemma \ref{872}, we again appeal to the equality \eqref{expressionofderivatives}. Recall that we have 
\[
\beta_0 =1,\, \beta_{\lambda'}\not=0, \; \beta_\ell=0 \text{ for } 1\leqslant \ell \leqslant \lambda' -1 \text{ and for }\lambda' +1\leqslant \ell \leqslant \lambda' +4.
\]
For 
\begin{equation}\label{dermonom}j\geqslant d-\lambda +1,
\end{equation} the derivative $f^{(j)} (t)$ is a monomial,
in particular we have the equality 
\begin{equation}\label{explicitderivative}
f^{(j)} (s) = \frac{d\, !}{(d-j)\, !}\cdot s^{d-j}.
\end{equation} 
We apply \eqref{expressionofderivatives} for four consecutive values of $j$, chosen in order that exactly two non zero $\beta_i$ are present on the RHS of these equalities. 
The corresponding indices are necessarily $i= 0$ and $i = \lambda'$. In this case \eqref{expressionofderivatives} becomes 
 \begin{equation}\label{fundbis}
 \frac{f^{(j)}(s)}{f(s)} =\frac{d\, !}{(d-j)\, !}\cdot \left( \frac {q}r\right)^j +\binom{d-\lambda'}{j-\lambda'} \cdot j\, ! \cdot 
\left( \frac qr\right)^j \cdot \frac{\beta_{\lambda'}}{q^{\lambda'}}\cdotp
\end{equation}
Hence the four values of $j$ are either 
\[
\lambda', \lambda' +1,\lambda' +2,\lambda' +3 \quad\text{or}\quad \lambda' +1,\lambda' +2,\lambda' +3, \lambda' + 4.
\]
 
This has to be compatible with \eqref{lambda+lambda} and \eqref{dermonom}: so we write
$j=d-\lambda+\ell$ with 
\begin{equation}\label{Equation:FourValues}
\text{ $\ell= \ell_0+i$, $(i=0,1,2,3)$ where } 
 \begin{cases}
 \displaystyle{
\ell_0 = 1 } & \hbox{ if $\lambda + \lambda' =d$ and if $\lambda + \lambda' =d+1$,}
 \\
 \displaystyle{
\ell_0 = 2}& \hbox{ if $\lambda + \lambda' =d+2$.}
\end{cases}
 \end{equation}
 
 We apply the formulas \eqref{fundbis}
with the four values of $\ell$ given in \eqref{Equation:FourValues}. 

\begin{equation}\label{system4*}
\frac{f^{(d-\lambda +\ell)}(s)}{f(s)} =\frac{d\, !}{(\lambda -\ell)\, !}\cdot \left( \frac {q}r\right)^{d-\lambda +\ell} +\binom{d-\lambda'}{d-\lambda-\lambda' +\ell} \cdot (d-\lambda +\ell)\, ! \cdot 
\left( \frac qr\right)^{d-\lambda +\ell} \cdot \frac{\beta_{\lambda'}}{q^{\lambda'}}\cdotp
\end{equation}
By \eqref{explicitderivative} 
the four equations of the system \eqref{system4*} 
become 
\begin{equation}\label{system3} 
s^{d} = \left( \left(\frac{qs}r\right)^{d-\lambda+\ell}+ \frac{(d-\lambda')\, !\cdot (d-\lambda+\ell)\, !}{(d-\lambda-\lambda'+\ell)\, !\cdot d\, !}\cdot
 \left( \frac {qs}r\right)^{d-\lambda +\ell} \cdot \frac{\beta_{\lambda'}}{q^{\lambda'}}\right) f(s).
\end{equation}

 \vskip .3cm 
\noindent
$\bullet$ {\bf We now prove that the case $q\neq 0$ and $s\neq 0$ is impossible.} 

To shorten notations, set $\nu=\lambda+\lambda'-d$, so that $\nu\in\{0,1,2\}$, and 
write $\tau =qs/r$, $\kappa = \beta_{\lambda'} / q^{\lambda'}$ (since $q\neq 0$), and 
 \begin{equation*}
 A_\ell : = \frac{(d-\lambda+\ell)\, !}{(\ell-\nu)\, !} \cdot \frac{(d-\lambda')\, !}{d\, !}, 
 \end{equation*}
 for the four values of $\ell$ given in \eqref{Equation:FourValues}. So \eqref{system3} 
 becomes 
 \begin{equation}\label{system3*}
s^{d} =\tau^{d-\lambda+\ell} \left( 1+ \kappa A_\ell \right) f(s).
\end{equation}

We now exploit the fact that $s\neq 0$.
We notice that \eqref{system3*} implies $\tau\not=0$ and $f(s)\not=0$, and also $1+\kappa A_\ell\neq 0$ for the four values of $\ell$ given in \eqref{Equation:FourValues}. 
We eliminate the variable $s^{d}$ among the four equations \eqref{system3*} and we obtain the three equalities, which are satisfied by the two unknowns 
 $\tau$ and $\kappa $, which both are $\neq 0$.
\[
\begin{cases}
\displaystyle{\tau = \frac{1+\kappa A_{\ell_0}}{1+\kappa A_{\ell_0+1}}}, \\
\\
\displaystyle{\tau = \frac{1+\kappa A_{\ell_0+1}}{1+\kappa A_{\ell_0+2}}},\\
\\
\displaystyle{\tau = \frac{1+\kappa A_{\ell_0+2}}{1+\kappa A_{\ell_0+3}}}\cdotp
\end{cases}
\]
We now write necessary and sufficient conditions to ensure that the two first equations are compatible and that the two last ones are compatible. We obtain the following system of two equations 
\[
\begin{cases}
(1+\kappa A_{\ell_0}) (1+\kappa A_{\ell_0+2}) = (1+ \kappa A_{\ell_0+1})^2\\
(1+\kappa A_{\ell_0+1}) (1+\kappa A_{\ell_0+3}) = (1+ \kappa A_{\ell_0+2})^2\\
\end{cases}
\]
which (since $\kappa\not=0$) is equivalent to the system 
\begin{equation}\label{Equation:system}
\left\{
\begin{aligned}
\bigl( A_{\ell_0+1}^2-A_{\ell_0}A_{\ell_0+2}\bigr)
\kappa 
&
= A_{\ell_0}+A_{\ell_0+2}-2A_{\ell_0+1},
\\
\bigl(A_{\ell_0+2}^2-A_{\ell_0+1}A_{\ell_0+3}\bigr)
\kappa 
&
= A_{\ell_0+1}+A_{\ell_0+3}-2A_{\ell_0+2}.
\end{aligned}
\right.
\end{equation}

If the system \eqref{Equation:system} has a solution then we have
\begin{equation}\label{Equation:569}
(A_{\ell_0+1}^2 -A_{\ell_0} A_{\ell_0+2})(A_{\ell_0 +1} +A_{\ell_0 +3} -2 A_{\ell_0 +2}) = (A_{\ell_0+2}^2 -A_{\ell_0+1} A_{\ell_0+3}) 
(A_{\ell_0 } +A_{\ell_0 +2} -2 A_{\ell_0 +1}). 
\end{equation}
 We factorize each $A_\ell$ as 
\[
A_\ell = (d-\lambda +1)\, ! \cdot \frac{(d-\lambda')\, !}{d\, !}A_\ell^*,
\]
with 
\begin{equation}\label{Equation:6.14}
A_\ell^* :=\prod_{i=2}^\ell \frac{\lambda' +i-\nu}{(i-\nu)^+},
\end{equation}
with the convention $x^+ = \max (x,\, 1)$.

By homogeneity, the equality \eqref{Equation:569} is equivalent to 
\begin{equation}\label{Equation:581}
(A_{\ell_0+1}^{*2} -A_{\ell_0}^* A_{\ell_0+2}^*)(A_{\ell_0 +1}^* +A_{\ell_0 +3}^* -2 A_{\ell_0 +2}^*) = (A_{\ell_0+2}^{*2} -A_{\ell_0+1} ^*A_{\ell_0+3}^*) 
(A_{\ell_0 }^* +A_{\ell_0 +2}^* -2 A_{\ell_0 +1}^*). 
\end{equation}
Actually, the equality \eqref{Equation:581} cannot hold. This is the purpose of the following lemma.

\begin{lemma} \label{Lemma:quotientQ} 
 Let $d$, $\lambda$ and $\lambda'$ be positive integers such that $d\leqslant \lambda +\lambda' \leqslant d+2$ and $\ell_0\in \{1,2\}$. Let $A_\ell^*$ $(1\leqslant \ell\leqslant 5)$ be defined by \eqref{Equation:6.14} with $\nu=\lambda+\lambda'-d$. Define
\begin{equation}\label{Equation:617}
Q:=\frac{
(A^*_{\ell_0}+A^*_{\ell_0+2}-2A^*_{\ell_0+1}) ( A_{\ell_0+2}^{*2}-A^*_{\ell_0+1}A^*_{\ell_0+3})}
{(A^*_{\ell_0+1}+A^*_{\ell_0+3}-2A^*_{\ell_0+2})( A_{\ell_0+1}^{*2}-A^*_{\ell_0}A^*_{\ell_0+2})}\cdotp 
\end{equation}
Then have 
\[ 
Q =
 \begin{cases} 
\displaystyle{ \frac{\lambda' +3}{3(\lambda' +2)}}&\text{ if } \lambda +\lambda' =d,\\
\\
 \displaystyle{\frac{ \lambda' +2}{2(\lambda' +1)}}&\text{ if } \lambda +\lambda' =d+1,\\
 \\
\displaystyle{\frac 12(\lambda'+2)} &\text{ if } \lambda +\lambda' =d+2.
 \end{cases}
\]
In these three cases, we have $Q \neq 1$.
\end{lemma}

\begin{proof}[Proof of Lemma \ref{Lemma:quotientQ}] We first suppose that $\lambda +\lambda' =d+2$, that is $\ell_0 =2$ and $\nu=2$. We have
the equalities
\[
A_1^*=1,\ 
A_2^* = \lambda', \ 
A_3^* = \lambda' (\lambda'+1), \ A_4^* = \frac{\lambda'(\lambda'+1) (\lambda'+2)}2, \ A_5^*= \frac{\lambda' (\lambda'+1) (\lambda'+2) (\lambda' +3)}6 
\cdotp
\] 
With these values, we obtain the equalities 
\[
A_2^* +A_4^* -2 A_3^* =\frac 12 \cdot \lambda'^2 (\lambda'-1),
\]
\[
 A_4^{*2}-A^*_3A^*_5= \frac 1{12} \, \lambda'^3 (\lambda'+1)^2 (\lambda '+2),
 \]
 \[
 A^*_3+A^*_5-2A^*_4=\frac 16 \lambda'^2 (\lambda'-1)(\lambda'+1)
 \]
 and
 \[
 {A^*_3}^2-A^*_2A^*_4=\frac 12 \lambda'^3 (\lambda'+1).
 \] 
 By \eqref{Equation:617}, we finally obtain 
 \[
 Q=\frac{
(A^*_2+A^*_4-2A^*_3) ( {A^*_4}^2-A^*_3A^*_5)}
{(A^*_3+A^*_5-2A^*_4)( {A^*_3}^2-A^*_2A^*_4)}
= \frac 12 (\lambda' +2).
 \]
This completes the proof of
Lemma \ref{Lemma:quotientQ} in that case $\lambda +\lambda '= d+2$.

We now suppose that $d\leqslant \lambda +\lambda'\leqslant d+1$, that is $\ell_0=1$ and $\nu\in\{0,1\}$. In that case, we can forget the symbol ${}^+$ in the definition \eqref{Equation:6.14}. 
We have the equalities
\[
\begin{gathered}
A_1^* = 1, \
A_2^* = \frac{\lambda'+2-\nu}{2-\nu}, \ 
A_3^* = \frac{(\lambda'+2-\nu)(\lambda'+3-\nu)}{(2-\nu)(3-\nu)}, \ 
\\
A_4^* = \frac{(\lambda'+2-\nu)(\lambda'+3-\nu)(\lambda'+4-\nu)}{(2-\nu)(3-\nu)(4-\nu)}\cdotp
\end{gathered}
\]
With these values, we obtain the equalities 
\[
A_1^* +A_3^* -2 A_2^* =\frac {\lambda'(\lambda'-1)}{(2-\nu)(3-\nu)},
\]
\[
 {A^*_3}^2-A^*_2A^*_4=\frac {\lambda'(\lambda'+2-\nu)(\lambda'+3-\nu)}{(2-\nu)^2(3-\nu)^2(4-\nu)},
 \]
 \[
 A^*_2+A^*_4-2A^*_3= \frac {\lambda'(\lambda'-1)(\lambda'+2-\nu)}{(2-\nu)(3-\nu)(4-\nu)}
 \]
 and 
\[
 A_2^{*2}-A^*_1A^*_3= \frac {\lambda' (\lambda'+2-\nu)}{(2-\nu)^2(3-\nu)}\cdotp
 \]
 By \eqref{Equation:617}, we finally obtain 
 \[
 Q=\frac{
(A^*_1+A^*_3-2A^*_2) ( {A^*_3}^2-A^*_2A^*_4)}
{(A^*_2+A^*_4-2A^*_3)( {A^*_2}^2-A^*_1A^*_3)}
= \frac{\lambda' +3-\nu}{(3-\nu)(\lambda'+2-\nu)}\cdotp
 \]
 Lemma \ref{Lemma:quotientQ} is proved also in that case.
 \end{proof}
 
 Thanks to Lemma \ref{Lemma:quotientQ}, this completes the proof that in Proposition \ref{central} the case $q\neq 0$ and $s\neq 0$ is impossible. 
 
 We continue the proof of Proposition \ref{central} in the other cases. 
 
 \vskip .3cm 
\noindent
$\bullet$ 
 {\bf We now prove that the case $q=0$ and $s\neq 0$ is impossible. }
 
Suppose that $q=0$. The second equality of \eqref{twopolynomials} implies that $\beta_{\lambda'+1} =0$. The equality \eqref{whenq=0} 
gives the vanishing of the derivative 
\begin{equation}\label{2757}f^{(\lambda'+1)} (s) =0.
\end{equation} 
By \eqref{lambda+lambda} we know that 
\[
d-\lambda +1 \leqslant \lambda'+1< d,
\]
where the last inequality comes from the assumption \eqref{1144}. This means that the derivative $f^{(\lambda'+1)} (t)$ is a monomial in $t$ with degree $\geqslant 1$. 
Combining with \eqref{2757} we obtain that $s=0$. Contradiction. 
 \vskip .3cm 
\noindent
$\bullet$ 
 {\bf We now prove that the case $q\neq 0$ and $s= 0$ is impossible. }
 
We benefit from the symmetry of the question to consider the homography $\mathfrak h_{s,r,q}$ which transforms $g$ to $f$. We also take into account the symmetry of the assumptions \eqref{1144} and \eqref{lambda+lambda} concerning these two polynomials. By the above alinea, we deduce that the case $q\neq 0$ and $s=0$ is also impossible.
 
\vskip .3cm 
\noindent
$\bullet$ {\bf The remaining case is $q=s=0$.} 
 
 We have $\mathfrak h_{0,r,0} (t) = r/t$. By Lemma \ref{998}
we deduce 
\[
[\mathfrak h_{0,r,0} (f)] (z) =\frac 1{f(0)} \cdot z^d f\left( \frac rz\right)
= \frac 1{\alpha_d}\cdot z^d f\left( \frac rz\right).
\] 
This relation gives the list of equalities
\[
\alpha_d \beta_j = \alpha_{d-j} r^{j}\ (0\leqslant j\leqslant d)
\]
after identification of the coefficients. In this relation, we fix $j= d-\lambda$. Since $\alpha_\lambda \neq 0$, we deduce that
$\beta_{d-\lambda} \neq 0$, which implies $d-\lambda \geqslant \lambda'$. This condition is compatible with 
\eqref{lambda+lambda} 
if and only if \eqref{1159} holds. By the definition of $\lambda$ and $\lambda'$ we deduce that $f$ and $g$ are both the sum of two or three monomials as indicated.
 
Recalling that $\alpha_0 = \beta_0 =1$, we obtain the conditions of the second item of Proposition 
 \ref{central}. 
This
completes the proof of this proposition.

\subsection{Proof of Theorem \ref{527} and Corollary \ref{Corollary2}} \label{Proof527et2}

 \begin{proof}[Proof of Theorem \ref{527}]
 Let $F$ and $G$ be two forms in $\cal E$ and $\gamma$ an homography such that $F\circ\gamma=G$. Our goal is to prove $F=G$. 
 
 Like in \S \ref{1033}, we consider the two monic polynomials $f$ and $g$ associated with $F$ and $G$ and the homography $ \mathfrak h= \tilde \gamma$ and we apply Proposition \ref{central}. The last assumption \eqref{Equation:trinome} of Theorem \ref{527} imply that $ \mathfrak h$ is of the form $\mathfrak h_{q,0}$ for some $q\in\K^\times$. By assumption $\cal E$ is $\K$--reduced, hence $\K$--dilation free, and therefore $F=G$.
 \end{proof}

\begin{proof}[Proof of Corollary \ref{Corollary2}]
Consider the first item of Corollary \ref{Corollary2}. The assumption $a_{d-1}=1$ implies that the set $\mathcal V^{(1)}_{d} (\K)$ does not contain binomials nor trinomials of the form \eqref{Equation:trinome}. 
According to Lemma \ref{Lemma:Wd1}, the set $\mathcal V^{(1)}_{d} (\K)$ is $\K$--dilation free. 
The first item now follows from Theorem \ref{527}.

For the proof of the second item, we use Lemma \ref{Lemma:rigid} together with Theorem \ref{527} in the same way as in the proof of the second item of Corollary \ref{Corollary2star} in section \ref{proof2.2}.
\end{proof}


 \vfill
 
\hbox{ 
 \vbox{ 
 \hbox{
\'Etienne Fouvry\\
}
 \hbox{
Univ. Paris--Saclay, CNRS \\
}
 \hbox{
Laboratoire de Math\' ematiques d'Orsay \\
}
 \hbox{
91405 Orsay, France \\
}
 \hbox{
E-mail: Etienne.Fouvry@universite-paris-saclay.fr}
}
\hfill 
 \vbox{
 \hbox{Michel Waldschmidt\\
}
 \hbox{
Sorbonne Universit\' e \\
}
 \hbox{
CNRS, IMJ--PRG \\
}
 \hbox{
75005 Paris, France\\
}
 \hbox{
E-mail: michel.waldschmidt@imj-prg.fr
}
}
}

 \vfill

 \vfill

 \vfill

 \end{document}